\documentclass[11pt]{article}
\usepackage{amsfonts}
\usepackage{amsthm}
\usepackage{amsmath}
\usepackage{amsfonts}
\usepackage{latexsym}
\usepackage{amssymb}
\usepackage[latin1]{inputenc}
\usepackage{verbatim}

\newtheorem{theorem}{Theorem}[section]
\newtheorem{lemma}[theorem]{Lemma}
\newtheorem{proposition}[theorem]{Proposition}

\newtheorem{remark}[theorem]{Remark}

\newtheorem{definition}[theorem]{Definition}
\newcommand\mloc{M_{\rm loc}}
\newcommand{\Omegawk}{\Omega_{\rm wk}}

\newcommand\re{\mbox{Re}\,}

\newcommand\NN{\mathbb{N}}

\title{Injective Envelopes and Local Multiplier Algebras of Some Spatial
Continuous Trace C$^*$-algebras\footnote{
This work is supported in part by the NSERC Discovery Grant Program (Canada) and CONICET (Argentina).}}

\author{Mart\'in Argerami, Douglas Farenick, and Pedro Massey
\\{ } \\
{\em
Department of Mathematics and Statistics,
University of Regina}\\
{\em Regina, Saskatchewan S4S 0A2, Canada}
\\ { } \\
{\em Departamento de Matem\'atica, Facultad de Ciencias Exactas} \\
{\em Universidad Nacional de La Plata,
Argentina} }

\begin{document}
\maketitle

\begin{abstract} A precise description of the injective envelope of a spatial continuous trace C$^*$-algebra $A$
over a Stonean space $\Delta$ is given. The description is based on the notion of a weakly
continuous Hilbert bundle, which we show herein to be a Kaplansky--Hilbert module over
the abelian AW$^*$-algebra $C(\Delta)$. We then use the description of the injective envelope of
$A$ to study the first- and second-order
local multiplier algebras of $A$. In particular, we show that the second-order local multiplier
algebra of $A$ is precisely the injective envelope of $A$.
\end{abstract}

%%%%%%%%%%%%%%%%%%%%%%%%%%%%
\section*{Introduction}

A commonly used technique in the theory of operators algebras is
to study a given C$^*$-algebra $A$ by one or more of its
enveloping algebras. Well known examples of such enveloping
algebras are the enveloping von Neumann algebra $A^{**}$ and the
multiplier algebra $M(A)$. In this paper we consider two others:
the local multiplier algebra $\mloc(A)$ and the injective envelope
$I(A)$, both of which have received considerable study and
application in recent years (see, for example,
\cite{ara-mathieu1999,birkenmeier-park-rizvi2009,blecher--paulsen2001,
dumitru--peligrad--visinescu2006,frank--paulsen2003,
pedersen1978,somerset1996,somerset2000}).

The C$^*$-algebras $\mloc(A)$ and $I(A)$ are difficult to determine precisely,
even for fairly rudimentary types of C$^*$-algebras $A$. For instance,
if we denote by $C_0(T)$ an abelian C$^*$-algebra and by $K(H)$ the ideal of
compact operators over $H$, their local multiplier algebra and injective envelope have
been readily computed;  but the injective envelope of
$C_0(T)\otimes K(H)$ is much more difficult to describe: see
\cite{hamana-tensor} for an abstract description and
\cite{ara-mathieu2008,argerami--farenick--massey2009} for a somewhat more
concrete one.

Our first goal in the present paper is to make a further
contribution to the issue of the determination of $I(A)$ and
$\mloc(A)$ from $A$ by considering continuous trace C$^*$-algebras
studied by Fell \cite{fell1961} that arise from continuous Hilbert
bundles. The class of such algebras contains in particular all C$^*$-algebras of
the form $C_0(T)\otimes K(H)$, which were studied in
\cite{argerami--farenick--massey2009}. Because the centres of $I(A)$ and
$\mloc(A)$ are AW$^*$-algebras, and thus have Stonean maximal ideal spaces,
we restrict ourselves in this paper to locally compact Hausdorff spaces $T$ that
are Stonean. In so doing, we establish an important first step toward
a complete analysis, in the case of non-Stonean $T$, of the C$^*$-algebras
$I(A)$, $\mloc(A)$, and $\mloc\left(\mloc(A)\right)$ for spatial continuous
trace C$^*$-algebras $A$ with spectrum $T$. As the passage from general $T$ to Stonean $T$
involves a number of technicalities, the application of the main results herein to
the case of arbitrary locally compact Hausdorff spaces $T$ will be deferred to a subsequent
article.

Our second goal  is to study and use the notion of a weakly continuous Hilbert
bundle $\Omegawk$ relative to a continuous Hilbert bundle $\Omega$ over a locally compact Hausdorff space $T$.
Particular cases of this notion have been previously considered in \cite{hamana-tensor,takemoto1976}.
It is natural to consider $\Omega$ as a C$^*$-module over the abelian C$^*$-algebra $C_0(T)$; if, moreover,
$T$ is a Stonean space $\Delta$, we then show  $\Omegawk$ carries the structure of a faithful AW$^*$-module over $C(\Delta)$.
In this latter situation, such C$^*$-modules are called Kaplansky--Hilbert modules.
We study the C$^*$-modules $\Omega$ and $\Omegawk$, as well as certain C$^*$-algebras of
endomorphisms of these modules,  using the beautiful machinery  Kaplansky developed
in his seminal work from the early 1950s \cite{kaplansky1953}.
In particular, we prove that the C$^*$-algebra $B(\Omegawk)$ of
bounded adjointable endomorphisms of $\Omegawk$ is the injective envelope and second-order local multiplier algebra of
the C$^*$-algebra $K(\Omega)$ of ``compact'' endomorphisms of $\Omega$.

Assuming that $T=\Delta$, a Stonean space, and in postponing the precise definitions until the
following section, we summarise in this paragraph the main results of the paper. In Section
\ref{section:omegawk}, we show that $\Omegawk$ is a Kaplansky--Hilbert module that contains
$\Omega$ as a C$^*$-submodule such that $\Omega^\bot=\{0\}$. In Section \ref{section:injective envelops},
we prove that $B(\Omegawk)$ is the injective envelope of both $K(\Omega)$ and the Fell continuous
trace C$^*$-algebra $A$ induced by the bundle $\Omega$. Section \ref{section:local multipliers} deals
with local multipliers, and we show that $B(\Omegawk)$ is the second-order local multiplier algebra
of both $K(\Omega)$ and Fell algebra $A$. We also prove that
the equality $\mloc(\mloc(A))=I(A)$ holds for certain type I non-separable C$^*$-algebras,
        generalising a result of Somerset \cite{somerset1996}.
        Finally, in Section \ref{section:reduction} we find that a
direct-sum decomposition of $\Omegawk$ leads to a corresponding
        decomposition of (the generally non-AW$^*$) algebra $\mloc(A)$
        but not to a decomposition of $A$.
%%%%%%%%%%%%%%%%%%%%%%%%%%%%%%%%%%%%%%%%%%%%%%%%%%
\section{Preliminaries}\label{section:preliminaries}

If $T$ is a locally compact Hausdorff space and $\{H_t\}_{t\in T}$ is family of Hilbert spaces,
 a vector field on $T$ with fibres $H_t$ is a function $\nu:T\rightarrow \bigsqcup_t\,H_t$ in which
$\nu(t)\in H_t$,
for every $t\in T$. Such a vector field $\nu$ is said to be bounded if the function
$t\mapsto \|\nu(t)\|$
is bounded. From this point on, the notation $T\rightarrow\bigsqcup_t\,H_t$ will be taken to also imply
that, for all $t$, the point $t$ is mapped into the corresponding fibre $H_t$.

\begin{definition}\label{definition:continuous Hilbert bundle}
A {\em continuous Hilbert bundle} \cite{dixmier--douady1963} is a triple $(T,\{H_t\}_{t\in T},\Omega)$,
where
$\Omega$ is a set of vector fields on $T$ with fibres $H_t$
such that:
\begin{enumerate}
\item[{\rm (I)}] $\Omega$ is a $C(T)$-module with the action $(f\cdot\omega)(t)=f(t)\omega(t)$;
\item[{\rm (II)}] for each $t_0\in T$, $\{\omega(t_0):\ \omega\in \Omega\}=H_{t_0}$;
\item[{\rm (III)}] the map $t\mapsto \|\omega(t)\|$ is continuous, for all $\omega\in\Omega$;
\item[{\rm (IV)}] $\Omega$ is closed under local uniform approximation---that is, if $\xi:T\rightarrow
                \bigsqcup_t\,H_t$
                is any vector field such that for every $t_0\in T$ and $\varepsilon>0$ there is an
                open set $U\subset T$
                containing $t_0$ and a $\omega\in \Omega$ with $\|\omega(t)-\xi(t)\|<\varepsilon$
                for all $t\in U$, then necessarily $\xi\in\Omega$.
\end{enumerate}
\end{definition}

Dixmier and Douady \cite{dixmier--douady1963} show that (I), (II), and (IV) can be replaced by
other axioms, such as those given by Fell \cite{fell1961},
without altering the structure that arises. For example, in the presence of the other axioms, (II)
is equivalent to ``$\{\omega(t_0)\,:\omega\in\Omega\}$ is dense in $H_{t_0}$, for each $t_0\in T$'';
in the presence of (IV), axiom (I) can be replaced by ``$\Omega$ is a complex vector space''.

We turn next to the notion of a weakly continuous Hilbert bundle.
If $(T,\{H_t\}_{t\in T},\Omega)$ is a continuous Hilbert bundle then,
by the polarisation identity, the function $t\mapsto \langle\omega_1(t),\omega_2(t)\rangle$ is
continuous for all $\omega_1,\omega_2\in\Omega$.
In defining $\langle\omega_1,\omega_2\rangle$ to be the map $T\rightarrow\mathbb C$ given by
$t\mapsto \langle\omega_1(t),\omega_2(t)\rangle$,
one obtains a $C(T)$-valued inner product on $\Omega$ which gives $\Omega$ the structure of an
inner product module
over $C(T)$.

\begin{definition}\label{definition:weakly continuous}
A vector field $\nu:T\rightarrow\bigsqcup_t\,H_t$ is said to be \emph{weakly continuous} with
respect to the continuous Hilbert bundle $(T,\{H_t\}_{t\in T},\Omega)$ if the function
\[
t\longmapsto \langle \nu(t),\omega(t)\rangle
\]
is continuous for all $\omega\in\Omega$. The set of all bounded weakly continuous vector fields
with respect to a given $\Omega$ will be denoted by $\Omegawk$, that is
\[
\Omegawk=\{\nu:T\rightarrow \bigsqcup_t\,H_t:\ \sup_t\|\nu(t)\|<\infty\ \mbox{ and }
\nu\mbox{ is weakly continuous}\}.
\]
We will call the quadruple  $(T,\{H_t\}_{t\in T},\Omega, \Omegawk)$
a \emph{weakly continuous Hilbert bundle over $T$}.
\end{definition}

We remark that when $T$ is compact,
$\Omegawk$ is a $C(T)$-module under the pointwise module action, and also
$\Omega\subset\Omegawk$  (because then every continuous field on $T$ is
bounded). However, the function $t\mapsto \langle\nu_1(t),\nu_2(t)\rangle$
is generally not continuous for arbitrary $\nu_1,\nu_2\in\Omegawk$.
Thus, although $\Omegawk$ is,
algebraically, a module over $C_b(T)$, it is not in general an inner product module over $C_b(T)$.
Nevertheless, if $T$ has the right topology---namely that of a Stonean space---then we
show (Theorem \ref{theorem:aw-mod1}) that it is possible to endow a weakly continuous Hilbert bundle with the structure of a
C$^*$-module over the C$^*$-algebra of continuous complex-valued functions on $T$.

The continuous trace C$^*$-algebras we consider herein were first studied by Fell \cite{fell1961}.
We now recall their definition.

Assume that $\{A_t\}_{t\in T}$ is
a family of C$^*$-algebras indexed by the locally compact Hausdorff topological space $T$. An operator
field is a map $a:T\rightarrow \bigsqcup_t\,A_t$ such that $a(t)\in A_t$, for each $t\in T$.

\begin{definition} Let $(T,\{H_t\}_{t\in T},\Omega)$ be a continuous Hilbert bundle.
An operator field $a:T\rightarrow\bigsqcup_{t\in T}K(H_t)$ is:
\begin{enumerate}
  \item \emph{almost finite-dimensional} (with respect to $\Omega$) if for each $t_0\in T$ and $\varepsilon>0$
       there exist an open set $U\subset T$ containing $t_0$
       and $\omega_1,\dots, \omega_n\in \Omega$ such that
          \begin{enumerate}
            \item[{(a)}] $\omega_1(t),\dots,\omega_n(t)$ are linearly independent for every $t\in U$, and
             \item[{(b)}] $\|p_ta(t)p_t-a(t)\|<\varepsilon$ for all $t\in U$, where $p_t\in B(H_t)$ is the
                         projection with range $\mbox{\rm Span}\,\{\omega_j(t)\,:\, 1 \le  j\le n\}$;
          \end{enumerate}
  \item \emph{weakly continuous} (with respect to $\Omega$) if the complex-valued function
\[
t\,\longmapsto\,\langle a(t)\omega_1(t), \omega_2(t)\rangle
\]
is continuous for every $\omega_1,\omega_2\in\Omega$.
\end{enumerate}
\end{definition}

\begin{definition}\label{definition:Fell algebra} {\rm (\cite{fell1961})}
Let $(T,\{H_t\}_{t\in T},\Omega)$ be a continuous Hilbert bundle.
The {\em Fell algebra of the Hilbert bundle}
$(T,\{H_t\}_{t\in T},\Omega)$, denoted by $A=A(T,\{H_t\}_{t\in T},\Omega)$,
is the set of all weakly continuous, almost
finite-dimensional operator fields $a:T\rightarrow\bigsqcup_{t\in
T}K(H_t)$ for which $t\mapsto \|a(t)\|$ is continuous and vanishes
at infinity, endowed with pointwise operations and norm
\[
\|a\|\,=\,\max_{t\in T}\,\|a(t)\|\,,\qquad a\in A\,.
\]
\end{definition}
We shall make repeated use of the following fact about the  Fell algebras of Hilbert bundles:
if $A=A(T,\{H_t\}_{t\in T},\Omega)$, for some
continuous Hilbert bundle $(T,\{H_t\}_{t\in T},\Omega)$, then
$A$ is a continuous trace
C$^*$-algebra with spectrum $\hat A\simeq T$ \cite[Theorems 4.4, 4.5]{fell1961}.

%%%%%%%%%%%%%%%%%%%%%%%%%%%%%%%%%%%%%%%%%%%%%%%%%%%
\section{An AW$^*$-module Structure for $\Omegawk$}\label{section:omegawk}

Assume henceforth that $T=\Delta$ is a Stonean space;
that is, $\Delta$ is Hausdorff, compact, and extremely
disconnected. The abelian C$^*$-algebra $C(\Delta)$ is an
AW$^*$-algebra and so one may ask whether the C$^*$-modules
$\Omega$ and $\Omegawk$ are AW$^*$-modules in the sense of
Kaplansky \cite{kaplansky1953}. We shall show that this is indeed
true for the module $\Omegawk$. As a consequence of this last fact we shall
get that the C$^*$-algebra $B(\Omegawk)$ of bounded adjointable endomorphisms
of $\Omegawk$ is an AW$^*$-algebra of type I.

The following lemmas are needed to describe the $C(\Delta)$-Hilbert module
structure of $\Omegawk$.

\begin{lemma}\label{lemma:norm off a meagre set}
Let $f:\Delta\rightarrow\mathbb{R}$ be a lower semicontinuous function such that there exist
$g\in C(\Delta)$ and a meagre set $M\subset\Delta$ with $f(s)=g(s)$ for all $s\in\Delta\setminus M$. Then
\[
\sup_{s\in\Delta}\, g(s)=\sup_{s\in\Delta\setminus M}f(s)=\sup_{s\in\Delta}\,f(s).
\]
\end{lemma}
\begin{proof}
Let $\rho=\displaystyle\sup_{s\in\Delta\setminus
M}\,f(s)=\sup_{s\in\Delta\setminus
M}\,g(s)\leq \displaystyle\sup_{s\in\Delta}\,g(s)$; then $f(s)\le \rho$ for all $s\in
\Delta\setminus M$.
Because $\Delta$ is a Baire space,  $\overline{\Delta\setminus M}=\Delta$; thus,
by the lower semi-continuity, $f(s)\le\rho$ for every
$s\in\Delta$. The same argument yields that $g(s)\le \rho$ for
all $s\in \Delta$.
\end{proof}

\begin{lemma}\label{lemma:lowersemicontinuous} Assume that $(\Delta, \{H_s\}_{s\in \Delta}, \Omega)$ is a
continuous Hilbert bundle and $\nu\in \Omegawk$. Then
\begin{enumerate}
  \item\label{lemma:lowersemicontinuous:1} the function $s\mapsto\|\nu(s)\|^2$ is lower semicontinuous;
  \item\label{lemma:lowersemicontinuous:2} there is a meagre subset $M\subset \Delta$ and a
        continuous function $h:\Delta\rightarrow\mathbb R_+$ such
        that
     \begin{enumerate}
         \item $h(s)=\|\nu(s)\|^2$ for all $s\in\Delta\setminus M$,
        and \item $\|h\|=\displaystyle\sup_{s\in\Delta\setminus
        M}\,\|\nu(s)\|^2=\displaystyle\sup_{s\in\Delta}\,\|\nu(s)\|^2$.
\end{enumerate}
\end{enumerate}
\end{lemma}

\begin{proof} Let $r\in\mathbb R$ be fixed and consider
$U_r=\{s\in\Delta\,:\,r<\|\nu(s)\|^2\}$. We aim to show that $U_r$
is open. Choose $s_0\in U_r$. Thus, $r<\|\nu(s_0)\|^2$. By
Parseval's formula, there are orthonormal vectors
$\xi_1,\dots,\xi_n\in H_{s_0}$ such that
$r<\displaystyle\sum_{j=1}^n|\langle\nu(s_0),\xi_j\rangle|^2\le\|\nu(s_0)\|^2$.
Choose any $\mu_1,\dots,\mu_n\in\Omega$ such that
$\mu_j(s_0)=\xi_j$, for each $j$. Because $\xi_1,\dots,\xi_n$ are
orthogonal, $\mu_1(s),\dots,\mu_n(s)$ are linearly independent in
an open neighbourhood of $s_0$. Hence, by \cite[Lemma
4.2]{fell1961}, there is an open set $V$ containing $s_0$ and
vector fields $\omega_1,\dots,\omega_n \in\Omega$ such that
$\omega_1(s),\dots,\omega_n(s)$ are orthonormal for all $s\in V$,
and $\omega_j(s_0)=\xi_j$ for each $j$. The function
\[
g(s)\,=\,\sum_{j=1}^n |\langle \nu(s), \omega_j(s)\rangle|^2
\]
on $\Delta$ is continuous and satisfies $g(s)\le\|\nu(s)\|^2$, for
every $s\in V$, and $r<g(s_0)$. Therefore, by the continuity of
$g$, there is an open set $W\subset V$ containing $s_0$ such that
$r<g(s)\le\|\nu(s)\|^2$ for all $s\in W$. This proves that $U_r$
contains an open set around each of its points. That is, $U_r$ is
open.

Because every bounded nonnegative lower semicontinuous function on
a Stonean space $\Delta$ agrees with a nonnegative
continuous function  off a meagre set
$M$ \cite[Proposition III.1.7]{Takesaki-bookI}, the function $h\in
C(\Delta)$ as in \eqref{lemma:lowersemicontinuous:2} exists and satisfies $h(s)=\|\nu(s)\|^2$ for
$s\in\Delta\setminus M$.

The last statement follows from Lemma \ref{lemma:norm off a meagre set}.
\end{proof}

Let $(\Delta,\{H_t\}_{t\in \Delta},\Omega, \Omegawk)$
be a weakly continuous Hilbert bundle over $\Delta$.
Given $\nu\in\Omegawk$, the function $h$ that arises in Lemma
\ref{lemma:lowersemicontinuous} will be denoted by $\langle\nu,\nu\rangle$. There
is no ambiguity in so doing because if $h_1,h_2\in C(\Delta)$ and
if $h_1(s)=h_2(s)$ for all $s\not\in(M_1\cup M_2)$ for some meagre
subsets $M_1$ and $M_2$, then $h_1$ and $h_2$ agree on $\Delta$.
(If not, then by continuity, $h_1$ and $h_2$ would differ on an
open set $U$; but $\emptyset\not=U\subset M_1\cup M_2$ is in
contradiction to the fact that no meagre set in a Baire space can
contain a nonempty open set.)

Now use the polarisation identity to define $\langle\nu_1,\nu_2\rangle\in C(\Delta)$
for any pair $\nu_1,\nu_2\in \Omegawk$. This gives $\Omegawk$ the structure of pre-inner
product module over $C(\Delta)$ whereby
for each $\nu_1,\nu_2\in\Omegawk$ there is a meagre subset $M_{\nu_1,\nu_2}\subset \Delta$
such that the continuous function $\langle \nu_1, \nu_2 \rangle$ satisfies
\[
\langle \nu_1, \nu_2 \rangle\,(s)\,=\,\langle\nu_1(s), \nu_2(s)\rangle\,,\qquad\forall\,
s\in \Delta\setminus M_{\nu_1,\nu_2}\,.
\]
In particular, if $\nu\in\Omegawk$ and $\omega\in\Omega$, then
\[
\langle \nu, \omega \rangle\,(s)\,=\,\langle\nu(s), \omega(s)\rangle\,,\qquad\forall\,s\in \Delta\,.
\]
In fact, $\Omegawk$ is an inner product module over $C(\Delta)$, for if $\nu\in\Omegawk$ satisfies $\langle\nu,\nu\rangle=0$,
then Lemma \ref{lemma:lowersemicontinuous} yields $\|\nu(s)\|^2=0$ for all $s\in\Delta$. Therefore,
\[
\|\nu\|\,=\,\|\langle\nu,\nu\rangle\|^{1/2}\,,\quad\nu\in\Omegawk\,,
\]
defines a norm
on $\Omegawk$, where
\begin{equation}\label{e:norm}
\|\nu\|^2\,=\,\sup_{s\in\Delta}\,\langle \nu(s),\nu(s)\rangle\,=\,\|\langle \nu,\nu\rangle\|\,.
\end{equation}

\medskip

Recall that given a C$^*$-algebra $B$, a \emph{Hilbert C$^*$-module over $B$} is a left $B$-module $E$
together with a $B$-valued definite sequilinear map $\langle\,,\,\rangle$ such
that $E$ is complete with the norm $\|\nu\|=\|\langle \nu,\nu\rangle\|^{1/2}$ (we refer to
\cite{Lance-book} for a detailed account on Hilbert modules).

Note that if $\nu\in \Omegawk$, then
$|\nu|(s):=\langle \nu,\nu\rangle^{1/2}(s)\geq \|\nu(s)\|$ for $s\in \Delta$
and there exists a meagre set $M\subset \Delta$ with
$|\nu|(s)=\|\nu(s)\|$ if $s\in (\Delta\setminus M)$ (Lemma \ref{lemma:lowersemicontinuous}).
These facts will be used repeatedly from now on.

\begin{proposition}
$\Omegawk$ is a C$^*$-module over $C(\Delta)$ and $\Omega$ is a
C$^*$-submodule of $\Omegawk$.
\end{proposition}

\begin{proof} The only Hilbert C$^*$-module axiom that is not obviously satisfied
by $\Omegawk$ is the axiom of completeness.
Let $\{\nu_i\}_{i\in\mathbb N}$ be a Cauchy sequence in
$\Omegawk$. By the equality \eqref{e:norm}, $\{\nu_i(s)\}_{i\in\mathbb N}$ is a
Cauchy sequence in $H_s$ for every $s\in\Delta$. Let $\nu(s)\in
H_s$ denote the limit of this sequence so that
$\nu:\Delta\rightarrow\bigsqcup_{s\in\Delta}H_s$ is a vector field.

Choose $\omega\in\Omega$ and consider the function
$g_{i,\omega}\in C(\Delta)$ given by $g_{i,\omega}(s)=\langle
\omega(s),\nu_i(s)\rangle$. Let $\varepsilon>0$. Then there is
$N_\varepsilon\in\mathbb N$ such that
$\|\nu_i-\nu_j\|<\varepsilon$, for all $i,j\ge N_\varepsilon$.
Therefore, the Cauchy-Schwarz inequality yields
\[
\sup_{s\in\Delta}\,|g_{i,\omega}(s)-g_{j,\omega}(s)|\,<\,\varepsilon\,\|\omega\|\,,\quad\forall\,i,j\ge N_\varepsilon\,.
\]
Thus, the sequence $\{g_{i,\omega}\}_i$ is Cauchy in $C(\Delta)$; let
$g_\omega\in C(\Delta)$ denote its limit.
Observe that
 $g_\omega(s)=\lim_i\langle\nu_i(s),\omega(s)\rangle=\langle\nu(s),\omega(s)\rangle$, for all $s\in\Delta$. As the
choice of $\omega\in\Omega$ is arbitrary, this shows that $\nu$ is weakly continuous.
The Cauchy sequence $\{\nu_i\}_{i\in\mathbb N}$ is necessarily uniformly bounded by,
say, $\rho>0$,
and then
$\|\nu(s)\|\le \rho$
for every $s\in\Delta$. That is, the function $s\rightarrow\|\nu(s)\|$ is bounded and so
$\nu\in\Omegawk$. Finally, if $i,j\ge N_\varepsilon$, then for any $s\in\Delta$ we have $\|\nu(s)-\nu_i(s)\|
\le\|\nu(s)-\nu_j(s)\|+\|\nu_j(s)-\nu_i(s)\|\le \|\nu(s)-\nu_j(s)\|+\varepsilon$, and so letting $j\rightarrow \infty$ yields
$\|\nu(s)-\nu_i(s)\|\le\varepsilon$ for every $s\in\Delta$. That is, $\|\nu-\nu_i\|\rightarrow 0$,
which proves that $\Omegawk$ is complete.

For the case of $\Omega$,
let $\{\omega_n\}_{n\in\mathbb N}$ be a Cauchy sequence in $\Omega$. For each
$s\in\Delta$, $\{\omega_n(s)\}_{n\in\mathbb N}$ is a Cauchy sequence in $H_s$; let
$\omega(s)$ denote the limit.
Since the limit is uniform, it is in particular locally uniform, and so $\omega\in\Omega$.
Hence, $\Omega$ is complete.
\end{proof}

\begin{definition}\label{definition:Kaplansky-Hilbert module}
A Hilbert C$^*$-module $E$ over a $C^*$-algebra $B$ is called
a {\em Kaplansky--Hilbert module} if in addition
$B$ is an abelian AW$^*$-algebra and
the following three
properties hold \cite[p.~842]{kaplansky1953} (Kaplansky's original term for such a module was
``faithful AW$^*$-module''):
\begin{enumerate}
  \item\label{definition:Kaplansky-Hilbert module:1} if $e_i\cdot \nu=0$ for some family $\{e_i\}_i\subset B$ of pairwise-orthogonal
                projections and $\nu\in E$, then also $e\cdot\nu=0$,
               where $e=\sup_i\,e_i$;
  \item\label{definition:Kaplansky-Hilbert module:2} if
$\{e_i\}_i\subset B$ is a family of pairwise-orthogonal
projections such that $1=\sup_i\,e_i$, and if $\{\nu_i\}_i\subset
E$ is a bounded family, then there is a $\nu\in E$
such that $e_i\cdot\nu=e_i\cdot\nu_i$ for all $i$;
  \item if $\nu\in E$, then $g\cdot\nu=0$ for all $g\in B$ only if $\nu=0$.
\end{enumerate}
\end{definition}

\begin{remark}\label{remark:about the sum notation}
The element $\nu\in E$ obtained in the situation described in (ii) will sometimes be denoted as
$\sum_ie_i\nu_i$. It should be emphasized that this is not a pointwise sum.
\end{remark}

\begin{theorem}\label{theorem:aw-mod1}
$\Omegawk$ is a Kaplansky--Hilbert module over $C(\Delta)$.
\end{theorem}

\begin{proof}
For property (i), assume that  $ \nu\in \Omegawk$ and
$\{e_i\}_i\subset C(\Delta)$ is a family of pairwise-orthogonal projections
with supremum $e\in C(\Delta)$ for which $e_i\cdot \nu=0$ for all $i$.
Because projections in $C(\Delta)$ are the characteristic functions of clopen sets,
there are pairwise-disjoint clopen sets $U_i\subset \Delta$ such that $e_i=\chi^{\phantom{a}}_{U_i}$. Thus,
for each $i$, using Lemma \ref{lemma:lowersemicontinuous},
\begin{align*}
0&=\|e_i\cdot\nu\|^2=\displaystyle\max_{s\in\Delta}\,\langle e_i\cdot\nu,e_i\cdot\nu\rangle(s)
=\sup_{s\in\Delta}\langle e_i(s)\nu(s),e_i(s)\nu(s)\rangle\\
&=
\displaystyle\max_{s\in\Delta}\,e_i(s)\,\left[\langle \nu, \nu\rangle(s) \right]
=\displaystyle\max_{s\in U_i}\,\langle \nu, \nu\rangle(s)\,,
\end{align*}
and so $\langle\nu,\nu\rangle(s)=0$ for every $s\in U_i$.
Let $U=\bigcup_iU_i$.
The set $\overline U$ is clopen and $\chi^{\phantom{a}}_{\overline U}
=\sup_i\,e_i=e$ \cite[\S 8]{Berberian-book}. As $\langle\nu,\nu\rangle$ is a continuous
function that vanishes on $U$, it also vanishes on $\overline U$. Hence,
\[
\|e\cdot\nu\|^2\,=\,\displaystyle\max_{s\in\Delta}\,e(s)\left[\langle \nu, \nu\rangle(s) \right]\,=\,
\displaystyle\max_{s\in\overline U}\,\langle \nu, \nu\rangle(s)\,=\,0\,,
\]
which yields property (i).

For the proof of property (ii), assume that $\{e_i\}_i\subset
C(\Delta)$ is a family of pairwise-orthogonal projections such
that $1=\sup_i\,e_i$ and that $\{ \nu_i\}_i\subset  \Omegawk$ is a
family such that $K=\sup\|\nu_i\|<\infty$; we aim to
prove that there is a $\nu\in\Omegawk$ such that $e_i\cdot
\nu=e_i\cdot  \nu_i$ for all $i$. As before,
assume that $e_i=\chi^{\phantom{a}}_{U_i}$ and $U=\bigcup_iU_i$. Then $1=\sup_i\,e_i$
implies that $\overline U=\Delta$.

For each $\omega\in\Omega$, consider the unique function
$f_\omega\in C(\Delta)$  such that $e_i\,f_\omega=e_i\,\langle
\omega,\nu_i\rangle$ for all $i$ (its existence guaranteed by the fact that $\Delta$ is the
Stone--\v Cech compactification of $U$). Note that for $s\in U_i$ we have
that $f_\omega(s)=\langle \omega(s),\nu_i(s)\rangle$. Hence,
$|f_\omega(s)|\leq K\,\|\omega(s)\|\, $ for $s\in U$; the same inequality holds
for all $s\in \Delta$ because $\overline U=\Delta$
and both sides of the inequality are continuous functions of $s$.
Moreover, if $\omega_1,\, \omega_2\in \Omega$ and $\alpha\in
\mathbb C$ then, for $s\in U$ we get that
$f_{\alpha\,\omega_1+\omega_2}(s)=\alpha \,
f_{\omega_1}(s)+f_{\omega_2}(s)$ and, therefore, that
$f_{\alpha\,\omega_1+\omega_2}=\alpha \,
f_{\omega_1}+f_{\omega_2}$. Thus, for each $s\in \Delta$ the
function $\omega(s)\mapsto f_\omega(s)$ is a well-defined, bounded linear functional on
$H_s$. Let $\nu(s)\in H_s$ be the representing vector for this
functional, yielding a vector field
$\nu:\Delta\rightarrow\bigsqcup_{s\in\Delta}H_s$. Since
$\langle\nu(s),\omega(s)\rangle=\overline{f_\omega(s)}$, for every
$\omega\in\Omega$, $\nu$ is weakly continuous. It remains to show
that $\nu$ is a bounded vector field. If $s\in U$,
\[
\|\nu(s)\|=\sup_{\omega\in\Omega , \|\omega(s)\|=1}|\langle \omega(s), \nu(s)\rangle|=
\sup_{\omega\in\Omega , \|\omega(s)\|=1}|f_\omega(s)|\leq \sup_i\|\nu_i\|=K\, ,
\]
which shows that $\|\nu(s)\|$ is uniformly bounded on $U$. Thus, since $U$ is dense,
the lower semicontinuous function $s\mapsto\|\nu(s)\|^2$ is bounded on $\Delta$.
Therefore, $\nu\in\Omegawk$.

Now we show that $e_i\cdot\nu=e_i\cdot \nu_i$, for all $i$. Fix $i$ and
$s\in U_i$ and consider $\omega\in\Omega$. Then,
\begin{align*}
\langle \omega(s),\,e_i(s)\,\nu(s)\rangle &= \langle \omega(s),\,\nu(s)\rangle= f_\omega(s)\\
&=e_i(s)\,f_\omega(s)=e_i(s) \langle \omega(s),\nu_i(s)\rangle \\
&= \langle \omega(s),e_i(s)\,\nu_i(s)\rangle\,.
\end{align*}
Since $(e_i\cdot\nu)(s)=0=(e_i\cdot \nu_i$)(s) for $s\in \Delta\setminus U_i$ we conclude that $e_i\cdot\nu=e_i\cdot \nu_i$.

For the proof of property (iii), assume that $ \nu\in \Omegawk$ satisfies $g\cdot \nu=0$ for all $g\in C(\Delta)$.
Then, in particular, $\langle\nu,\nu\rangle\cdot \nu=0$, so $\langle\nu,\nu\rangle=0$.
Hence, from $\| \nu\|=\|\langle\nu,\nu\rangle\|^{1/2}=0$ we conclude that $\nu=0$.
\end{proof}

%%%%%%%%%%%%%%%%%%%%%%%%%%%%%%%%%%%%%%%%%%%%%%%%%%%
\section{Endomorphisms of $\Omega$ and $\Omegawk$}\label{section:injective envelops}
Throughout this section $A$ will denote the Fell C$^*$-algebra of the continuous
Hilbert bundle $(\Delta,\{H_s\}_{s\in \Delta},\Omega)$, as
described in Definition \ref{definition:Fell algebra}, with $\Delta$ Stonean.
Let $B(\Omega)$ and $B(\Omegawk)$ denote, respectively, the
C$^*$-algebras of adjointable $C(\Delta)$-endomorphisms of
$\Omega$ and $\Omegawk$. Since, by Theorem \ref{theorem:aw-mod1}, $\Omegawk$ is a
Kaplansky--Hilbert AW$^*$-module over $C(\Delta)$,
$B( \Omegawk)$ coincides with the set of all $C(\Delta)$-endomorphisms of
$\Omegawk$ \cite[Theorem 6]{kaplansky1953} and is a type I
AW$^*$-algebra with centre $C(\Delta)$ \cite[Theorem
7]{kaplansky1953}.

In the particular case where $\Omega$ is given by
the trivial Hilbert bundle $(\Delta,\{H\}_{s\in \Delta},C(\Delta,H))$
with $H$ is a fixed Hilbert space, Hamana \cite{hamana-tensor}
proved that $B(\Omegawk)\cong C(\Delta)\overline\otimes B(H)$,  the
monotone complete tensor product of $C(\Delta)$ and $B(H)$.

For each $ \nu_1, \nu_2\in \Omegawk$, consider the endomorphism
$\Theta_{ \nu_1, \nu_2}$ on $ \Omegawk$ defined by
\[
\Theta_{ \nu_1, \nu_2}\,( \nu)\,=\,\langle  \nu,  \nu_2\rangle\cdot \nu_1\,,\quad \nu\in  \Omegawk\,.
\]
For a Hilbert bundle $\Omega_0$, let
\[
F(\Omega_0)=\left\{\sum_{j=1}^n \Theta_{\omega_j,\omega_j'}\,:\,n\in\NN,\,\omega_j,\omega_j'\in \Omega\right\}.
\]
We will consider both $F(\Omega)$ and $F(\Omegawk)$.

If $\omega_1,\omega_2\in\Omega$, then $\Theta_{\omega_1,\omega_2}(\omega)\in\Omega$
for all $\omega\in\Omega$, and so
$F(\Omega)\subset B(\Omega)$. In fact, $F(\Omega)$ and $F(\Omegawk)$ are algebraic ideals in $B(\Omega)$ and
$B(\Omegawk)$ respectively. The norm-closures of these algebraic ideals,
namely $K(\Omega)$ and $K(\Omegawk)$,
are essential ideals in each of $B(\Omega)$ and $B(\Omegawk)$---called the ideals of
compact endomorphisms---and the multiplier algebras of $K(\Omega)$ and $K(\Omegawk)$
are, respectively, $B(\Omega)$ and $B(\Omegawk)$ \cite{Lance-book}.

When referring to rank-1 operators $x$ acting on a Hilbert space
$H$, we will use the notation $x=\xi\otimes\eta$ for such an
operator---the action on $\gamma\in H$ given by
$\gamma\mapsto\langle\gamma,\eta\rangle\,\xi$---and we reserve the
notation $\Theta_{ \xi, \eta}$ for ``rank-1'' operators acting on
a Hilbert module.

The term ``homomorphism'' will be used to mean a $*$-homomorphism
between C$^*$-algebras.

For any C$^*$-algebra $B$, we denote the injective envelope \cite{hamana1979a},
\cite[Chapter 15]{Paulsen-book} of $B$ by $I(B)$ (and we consider $I(B)$ as a C$^*$-algebra rather than as an
operator system).

The main result of the present section is the following.

\begin{theorem}\label{theorem:inclusion}
There exist C$^*$-algebra embeddings such that
\begin{equation}\label{theorem:inclusion:1}
K(\Omega)\,\subset\,A\,\subset\,B(\Omega)\,\subset\,B(\Omegawk)\,=\,I(K(\Omega))\,.
\end{equation} In particular, $I(K(\Omega))=I(A)=I(B(\Omega))=B(\Omegawk)$.
\end{theorem}

The proof of Theorem \ref{theorem:inclusion} and a description of
the inclusions in \eqref{theorem:inclusion:1} begin with the
following set of results.

\begin{lemma}\label{lemma:action of A}For every $a\in A$ and $\omega\in\Omega$, the
vector field $a\cdot \omega$ defined by $a\cdot \omega(s)=a(s)\omega(s)$
is an element of $\Omega$.
\end{lemma}
\begin{proof}
Let $a\in A$. Then $a^*a\in A_+$ and since all fields in $A$ are
weakly continuous, for every $\omega\in \Omega$ the map
$s\mapsto\|a(s)\omega(s)\|=\langle a^*a\cdot\omega(s),\omega(s)\rangle^{1/2}$ is continuous.

Suppose $s_0\in \Delta$ and $\varepsilon>0$. Because $H_{s_0}=\{\mu(s_0)\,:\,\mu\in\Omega\}$, there is a
$\mu\in\Omega$ such that $a(s_0)\omega(s_0)=\mu(s_0)$. Since
\[
\|a\cdot\omega(s)-\mu(s)\|^2\,=\,\|a(s)\omega(s)\|^2+\|\mu(s)\|^2-
2\re \langle a(s)\omega(s), \mu(s)\rangle
\]
is continuous on $\Delta$ and vanishes at $s_0$, there is an open set $U\subset\Delta$ containing $s_0$ such that
$\|a\cdot\omega(s)-\mu(s)\|<\varepsilon$ for all $s\in U$. As $\Omega$ is closed under local uniform approximation,
this proves that $a\cdot\omega\in\Omega$.
\end{proof}

\begin{proposition}\label{proposition:representation Lemma}  The map $\varrho:A\rightarrow B(\Omega)$
given by $\varrho(a)\omega=a\cdot\omega$, for $a\in A$ and $\omega\in \Omega$ is an
isometric homomorphism.
Furthermore, $K(\Omega)\subset \varrho(A)\subset B(\Omega)$ as C$^*$-algebras.
\end{proposition}

\begin{proof} It is clear that $\varrho$ is a homomorphism, and so we only need to
verify that it is one-to-one. To this end, assume that $\varrho(a)=0$.
Thus, $a(s)\omega(s)=0$ for every $\omega\in\Omega$ and every
$s\in\Delta$.
Because $H_s=\{\omega(s)\,:\,\omega\in\Omega\}$, this implies
that $a(s)=0$ for all $s\in\Delta$, and so $a=0$.

To show $K(\Omega)\subset \varrho(A)\subset B(\Omega)$ as C$^*$-algebras, consider $\Theta_{
\omega_1, \omega_2}$ with $\omega_1,\omega_2\in\Omega$. The map $s\mapsto\|\Theta_{\omega_1(s),
\omega_2(s)}\|$ is continuous because $\|\Theta_{\omega_1(s),\omega_2(s)}\|$ $=\,\|\omega_1(s)\|\,
\|\omega_2(s)\|$. For any
$\eta_1,\eta_2\in\Omega$, the map
\[
\langle\Theta_{\omega_1,\omega_2}\cdot\eta_1,\eta_2\rangle(s)
=\langle\eta_1,\omega_2\rangle(s)\,\langle\omega_1,\eta_2\rangle(s)
=\langle\eta_1(s),\omega_2(s)\rangle\,\langle\omega_1(s),\eta_2(s)\rangle
\]
is continuous. So $\Theta_{\omega_1,\omega_2}$ is also finite dimensional and weakly continuous,
which shows that $\Theta_{ \omega_1, \omega_2}\in A$ and
$K(\Omega)\subset \varrho(A)$.
\end{proof}

\begin{lemma}\label{lemma:density} With respect to the inclusion
$\Omega\subset\Omegawk$, we have $\Omega^\perp =\{0\}$.
\end{lemma}
\begin{proof}
Let $\nu\in \Omegawk$ be such that $\langle \nu,\omega\rangle=0$,
for every $\omega\in \Omega$. That is, for every $\omega\in
\Omega$ and for every $s\in \Delta$, $\langle
\nu(s),\omega(s)\rangle=0$. If $\nu\neq 0$,
there exists $s_0\in \Delta$ such that $\nu(s_0)\neq 0$.
By axiom (II) in Definition \ref{definition:continuous Hilbert bundle},
there exists $\omega\in \Omega$ such that
$\omega(s_0)=\nu(s_0)$, in contradiction to $\langle
\nu(s_0),\omega(s_0)\rangle=0$.
\end{proof}

\begin{lemma}\label{lemma:vector field with fixed value}
If $t_0\in \Delta$ and $\xi\in H_{t_0}$, then there exists $\omega\in
\Omega$ such that $\omega(t_0)=\xi$ and $\|\omega\|=\|\xi\|$.
\end{lemma}
\begin{proof}
The case $\xi=0$ is trivial. So assume that $\|\xi\|>0$. Let $\omega'\in\Omega$ with $\omega'(t_0)=\xi$.
Fix a clopen neighbourhood $V$ of $t_0$ such that
$V\subset \{t\in T\,:\,\|\omega'(t)\|\,\ge\,\|\omega'(t_0)\|/2\}$.
Let $h'(\cdot)=\|\xi\|\cdot\|\omega'(\cdot)\|^{-1}\in C(V)$; then $h'$
extends to a continuous function $h\in C(\Delta)$ with $h|_{\Delta\setminus V}=0$.
It is now straightforward to show that $\omega=h\cdot \omega'\in \Omega$ has the desired properties.
\end{proof}

\begin{proposition}\label{proposition:b(Omega) in B(Omegawk)} There exists an isometric homomorphism
$\vartheta: B(\Omega)\rightarrow B(\Omegawk)$ such that for $a\in A,\ \nu\in \Omegawk$,
 \begin{equation}\label{equation:action of A on Omegawk}
(\vartheta(\varrho(a))\nu)(s)=a(s) \nu(s),\ \ s\in \Delta\,.
\end{equation}
\end{proposition}

\begin{proof} Assume that $b\in B(\Omega)$ and $\omega\in \Omega$, $s\in\Delta$.
By Lemma \ref{lemma:vector field with fixed value},
 \begin{align*}
 \|(b\,\omega)(s)\|&=
 \displaystyle\sup_{\xi\in H_s,\,\|\xi\|=1}|\langle (b\,\omega)(s),\xi\rangle| =
 \displaystyle\sup_{\eta\in \Omega,\,\|\eta\|=1}|\langle (b\,\omega)(s),\eta(s)\rangle| \\
 &=\sup_{\eta\in \Omega,\,\|\eta\|=1}|\langle b\,\omega,\eta\rangle(s)|
= \displaystyle\sup_{\eta\in \Omega,\,\|\eta\|=1}|\langle \omega(s),(b^*\eta)(s)\rangle|\\
&\leq \|\omega(s)\| \,\sup_{\eta\in \Omega,\,\|\eta\|=1}\|b^*\eta\|
\leq  \|\omega(s)\| \,\|b^*\|=\|\omega(s)\| \,\|b\|\,.
\end{align*}
Therefore the function $\omega(s)\mapsto (b\,\omega)(s)$ is well
defined and induces a bounded linear operator $b(s)\in B(H_s)$
such that $(b\,\omega)(s)=b(s)\,\omega(s)$, for $s\in \Delta$ and
$\omega\in \Omega$, with $\sup_{s\in \Delta}\|b(s)\|\leq \|b\|$.
Moreover,
\begin{align*}
\|b\|&=\sup_{\|\omega\|=1}\,\|b\cdot\omega\|=\sup_{\|\omega\|=1}\,\sup_s\,\|b\cdot\omega(s)\|
=\sup_{\|\omega\|=1}\,\sup_s\,\|b(s)\omega(s)\|\\
&\leq\sup_{\|\omega\|=1}\,\sup_s\,\|b(s)\|\,\|\omega(s)\|
\leq\sup_s\|b(s)\|\leq\|b\|\, ,
\end{align*}
and so $\sup_{s\in \Delta}\|b(s)\|= \|b\|$.
Suppose now that $\nu\in \Omegawk$ and $s\in \Delta$, and define a vector field $\vartheta b\nu$
by $(\vartheta b\,\nu)(s)=b(s)\,\nu(s)$.
If $\eta\in \Omega$, then
\[
\langle(\vartheta b\,\nu)(s),\eta(s)\rangle\,=\,\langle \nu(s),b(s)^*\eta(s)\rangle\,
=\,\langle \nu(s),(b^*\eta)(s)\rangle\,
\]
is continuous, which shows that
$\vartheta b\,\nu$ is weakly continuous with respect to $\Omega$. Since $\vartheta b\,\nu$
is also uniformly bounded, we conclude that $\vartheta b\,\nu\in \Omegawk$. It is straightforward
to show that the map
$\nu\mapsto \vartheta b\,\nu$ is a bounded $C(\Delta)$-endomorphism of $\Omegawk$ and hence it gives rise
to an element $\vartheta b\in B(\Omegawk)$. It is clear that $\vartheta$ is a
homomorphism. If $\vartheta b=0$, then $b(s)\omega(s)=0$ for all $\omega\in\Omega$, $s\in\Delta$
and so $b(s)=0$ for all $s$; then $\|b\|=\sup_s\|b(s)\|=0$, and $b=0$.
So $\vartheta$ is one-to-one, and thus isometric.
Finally, it is clear that \eqref{equation:action of A on Omegawk} holds by construction.
\end{proof}

One consequence of the proof of Proposition \ref{proposition:b(Omega) in B(Omegawk)}
is that for every $b\in B(\Omega)$ there exists an operator field
$\{b(s)\}_{s\in \Delta}$ acting on the Hilbert bundle $\{H_s\}_{s\in \Delta}$
such that $(b\,\omega)(s)=b(s)\,\omega(s)$, for every $s\in \Delta$. This property, however, is not shared by all
elements of $B(\Omegawk)$.

\begin{lemma}\label{lemma:faithful compression}
If $z\in B(\Omegawk)$ and $\Theta_{\omega,\omega}z\Theta_{\mu,\mu}=0$ for all $\omega,\mu\in\Omega$,
then $z=0$.
\end{lemma}
\begin{proof}
For any $\xi,\,\omega,\,\mu\in\Omega$ we have that
\[
0=\Theta_{\omega,\omega}\,z\,\Theta_{\mu,\mu}\,\xi=\langle\xi,\mu\rangle\,\langle z\mu,\omega\rangle\,\omega.
\]
Hence, we get that
\[
0=\langle\xi,\mu\rangle\,|\langle z\mu,\omega\rangle|^2
=\langle\xi,\mu\rangle\,|\langle \mu,z^*\omega\rangle|^2.
\]
We are free to choose $\xi,\mu\in\Omega$. Fix $s$, and
choose $\mu$ with $\mu(s)=z^*\omega(s)$; let
$\xi=\mu$.
Then, as $\mu\in\Omega$, we get $0=\langle \mu,\mu\rangle(s)=\langle \mu(s),\mu(s)\rangle$, so
$z^*\omega(s)=\mu(s)=0$. As $s\in \Delta$ is arbitrary, $z^*\omega=0$
for every $\omega\in \Omega$. For any $\nu\in \Omegawk$ and every $\omega\in \Omega$,
$\langle z\nu,\omega\rangle=\langle \nu,z^*\omega\rangle=0$.
By Lemma \ref{lemma:density} we conclude that $z\nu=0$ for $\nu\in \Omegawk$ and
hence $z=0$.
\end{proof}

\begin{proof}[Proof of Theorem \ref{theorem:inclusion}]

We consider the embeddings $A\xrightarrow{\varrho}B(\Omega)$ and $B(\Omega)\xrightarrow{\vartheta}B(\Omegawk)$
defined in Propositions \ref{proposition:representation Lemma} and \ref{proposition:b(Omega) in B(Omegawk)}.
In this way, we get the inclusions in \eqref{theorem:inclusion:1}.

Because $B(\Omegawk)$ is a type I AW$^*$-algebra, it is injective \cite[Proposition 5.2]{hamana1981}.
To show that $B(\Omegawk)$
is the injective envelope $I(K(\Omega))$ of $K(\Omega)$,
we need to show that the embedding $ \vartheta\circ\varrho$
of $K(\Omega)$ into $B(\Omegawk)$ is rigid \cite[Theorem 15.8]{Paulsen-book}:
that is, we aim to prove that if $\phi:B(\Omegawk)\rightarrow B(\Omegawk)$ is a
unital completely positive linear map for which $\phi|_{K(\Omega)}=\mbox{\rm id}_{K(\Omega)}$,
then $\phi=\mbox{\rm id}_{B(\Omegawk)}$.

Let $\phi:B(\Omegawk)\rightarrow B(\Omegawk)$ be such a ucp map
with $\phi|_{K(\Omega)}=\mbox{\rm id}_{K(\Omega)}$.
Suppose that $z\in B(\Omegawk)$ and $\omega,\mu\in\Omega$. Then
$\Theta_{ \omega, \omega}z\Theta_{ \mu, \mu}\,
=\,\Theta_{ \langle z\mu,\omega\rangle  \omega,\mu}\,\in K(\Omega)$.
Because $K(\Omega)$ is in the
multiplicative domain of $\phi$, we have that
$\phi(axb)=a\phi(x)b$ for all $x\in B(\Omegawk)$ and $a,b\in K(\Omega)$. This implies
that
\[
\Theta_{ \omega, \omega}\phi(z)\Theta_{ \mu,\mu}
=\phi(\Theta_{ \omega, \omega}z\Theta_{ \mu,\mu})
=\phi(\Theta_{ \langle z\mu,\omega\rangle  \omega,\mu})
=\Theta_{ \langle z\mu,\omega\rangle  \omega,\mu}
=\Theta_{ \omega, \omega}z\Theta_{ \mu,\mu},
\]
and so $\Theta_{ \omega, \omega}(z-\phi(z))\Theta_{ \mu,\mu}=0$.
Since $\omega,\mu$ were arbitrary,
Lemma \ref{lemma:faithful compression} implies that $z-\phi(z)=0$ and so
$\phi=\mbox{id}_{B(\Omegawk)}$.

We have shown above that
the inclusion $K(\Omega)\subset B(\Omegawk)$ is rigid. Moreover,
$K(\Omega)$ is an essential ideal of $B(\Omega)$ and
$K(\Omega)\subset A\subset B(\Omega)$. Hence, $I(K(\Omega))=I(A)=I(B(\Omega))=B(\Omegawk)$.
\end{proof}

We conclude this section with a remark about the ideal
$K(\Omegawk)$ of $B(\Omegawk)$. In type I AW$^*$-algebras, the ideal
generated by the abelian projections has a prominent role. As it happens, $K(\Omegawk)$
is precisely this ideal.

\begin{proposition}\label{proposition:k-omegawk} The C$^*$-algebra $K(\Omegawk)$ coincides with the ideal $J\subset B(\Omegawk)$
generated by the abelian projections of $B(\Omegawk)$. So
$K(\Omegawk)$ is a liminal C$^*$-algebra with Hausdorff spectrum.
\end{proposition}

\begin{proof}
By \cite[Lemma 13]{kaplansky1953}, a projection $e\in B(\Omegawk)$ is abelian if and only if there exists $\nu\in \Omegawk$
such that $|\nu|$ is a projection in $C(\Delta)$ and $e=\Theta_{\nu,\nu}$. Hence, $J\subset K(\Omegawk)$.

To show that $K(\Omegawk)\subset J$,
assume $\nu\in \Omegawk$ is nonzero. Let $\varepsilon>0$. We will show that there is an $x_\varepsilon\in J$
such that $\|\Theta_{\nu,\nu}-x_\varepsilon\|<\varepsilon$.
Let $V\subset\Delta$ be the (clopen) closure of
$\{s\in\Delta:\ |\nu|(s)<\varepsilon^{1/2}\}$, $U=\Delta\setminus V$ (also clopen)
and let $g=(1/|\nu|)\,\chi_U\in C(\Delta)_+$. Then $g|\nu|=\chi_U$ and $\|\,\chi_{\Delta\setminus
U}|\nu|\,\|<\varepsilon^{1/2}$.
Let $\nu'=g\cdot\nu$ so that $|\nu'|=\chi_U$. Hence, $\Theta_{\nu',\nu'}\in J$ and $\Theta_{\nu',\nu'}=g^2\cdot\Theta_{\nu,\nu}$.
Let $x_\varepsilon=|\nu|^2\cdot\Theta_{\nu',\nu'}\in J$. Then
\[
x_\varepsilon=|\nu|^2\cdot\Theta_{\nu',\nu'}=|\nu|^2\,g^2\,\Theta_{\nu,\nu}
=\chi_U\,\Theta_{\nu,\nu},
\]
and $x_\varepsilon-\Theta_{\nu,\nu}=\chi_{\Delta\setminus U}\cdot\Theta_{\nu,\nu}$.
Then
\begin{align*}
\|x_\varepsilon-\Theta_{\nu,\nu}\|&=
\sup_{\eta\in(\Omegawk)_1}\|\chi_{\Delta\setminus U}\cdot\Theta_{\nu,\nu}\,\eta\|
=\sup_{\eta\in(\Omegawk)_1}\|\chi_{\Delta\setminus U}\cdot\langle\eta,\nu\rangle\,\nu\|\\
&=\sup_{\eta\in(\Omegawk)_1}\,\max_{s\in\Delta\setminus U}|\langle\eta,\nu\rangle(s)|\,\|\nu(s)\|\\
&\leq\sup_{\eta\in(\Omegawk)_1}\,\max_{s\in\Delta\setminus U}|\eta|(s)\,|\nu|(s)|\,\|\nu(s)\|
\leq\max_{s\in\Delta\setminus U}|\nu|(s)^2<\varepsilon.
\end{align*}
As $\varepsilon$ was arbitrary and $J$ is closed,
we conclude that $\Theta_{\nu,\nu}\in J$. The polarisation
identity then shows that $\Theta_{\nu_1,\nu_2}\in J$ for all $\nu_1,\nu_2\in\Omegawk$.
Hence, $F(\Omegawk)\subset J$, and so
$K(\Omegawk)\subset J$.

It remains to justify the last assertion in the statement.
By the main result of \cite{halpern1966}, the ideal generated by the abelian projections in a type I
 AW$^*$-algebra is liminal and has Hausdorff spectrum. Hence, this is true of $K(\Omegawk)$.
\end{proof}

%%%%%%%%%%%%%%%%
\section{Multiplier and Local Multiplier Algebras}\label{section:local multipliers}

In the previous section we established the inclusions
$K(\Omega)\subset A\subset B(\Omega)\subset B(\Omegawk)$, as
C$^*$-subalgebras, and we showed that $I(A)=B(\Omegawk)$. The
present section refines these inclusions to incorporate multiplier
algebras and local multiplier algebras.

Given a C$^*$-algebra $C$, we denote by $M(C)$ and $\mloc(C)$ its multiplier
and local multiplier algebra \cite{Ara--Mathieu-book} respectively.

The second order
local multiplier algebra of $C$ is $\mloc\left(\mloc(C)\right)$, the local multiplier algebra of $\mloc(C)$.
By \cite[Corollary 4.3]{frank--paulsen2003}, the local multiplier
algebras (of all orders) of $C$ are C$^*$-subalgebras of the
injective envelope $I(C)$ of $C$. In particular, $C\subset \mloc(C)\subset \mloc\left(\mloc(C)\right)\subset
I(C)$ as C$^*$-subalgebras.

By a well known theorem of Kasparov \cite[Theorem
1.2.33]{Ara--Mathieu-book}, \cite[Theorem 2.4]{Lance-book},
$M(K(\Omega))=B(\Omega)$. We remark that all the subalgebras we consider are
essential in $B(\Omegawk)$ (i.e. the annihilator is zero), and so whenever we write
$M(C)$ for one of these subalgebras $C\subset B(\Omegawk)$, we mean the
concrete realization \cite{pedersen1984}
\[
M(C)=\{x\in B(\Omegawk):\ xC+Cx\subset C\}.
\]

The following theorem is the main result of this section.

    \begin{theorem}\label{theorem:extension of Somerset}
    With the notations from the previous sections, we have the equality $\mloc(A)=\mloc(K(\Omega))$ and
the following inclusions (as C$^*$-subalgebras):
\begin{align}\label{e:loc-loc2}
\nonumber M(A) &\subset M(K(\Omega))=B(\Omega) \\
&\subset\,\mloc(K(\Omega)) \subset\,\mloc\left(\mloc(K(\Omega))\right)= B(\Omegawk)\,.
\end{align}
In particular, $\mloc\left(\mloc(A)\right)=I(A)$.
\end{theorem}

Ara and Mathieu have presented examples of Stonean spaces $\Delta$ and trivial
Hilbert bundles $\Omega$
where the inclusion $\mloc(K(\Omega)) \subset
\mloc\left(\mloc(K(\Omega))\right)$ in \eqref{e:loc-loc2} is proper
\cite[Theorem 6.13]{ara-mathieu2008}. As a consequence of Theorem
\ref{theorem:extension of Somerset} and the fact that $B(\Omegawk)=I(K(\Omega))$,
we see that this gap cannot occur for
higher local multiplier algebras, i.e. for all $k\geq 2$,
$ \mloc^{k+1}(K(\Omega))= \mloc^k(K(\Omega))$  ---
where $\mloc^{k+1}(K(\Omega))=\mloc(\mloc^k (K(\Omega)))$ for $k\geq 1$.

\bigskip

The proof of Theorem \ref{theorem:extension of Somerset} is achieved through a number of lemmas.

\begin{lemma}\label{lemma:pos dense}
The set
\[
F_+=\{\displaystyle\sum_{j=1}^n\Theta_{\omega_j,\omega_j}\,:\,
n\in\mathbb N,\,\omega_j\in\Omega\}
\]
is dense in the positive cone of $K(\Omega)$.
\end{lemma}
\begin{proof} Assume that $h\in K(\Omega)_+$ and let $\varepsilon>0$ be arbitrary.
For each $s_0\in\Delta$ consider the
positive compact operator $h(s_0)\in K(H_{s_0})$. Then there are
vectors $\xi_1,\dots,\xi_{n_{s_0}}\in H_{s_0}$ such that
$$\|h(s_0)-\displaystyle\sum_{j=1}^{n_{s_0}}\xi_j\otimes\xi_j\|<\varepsilon\,.$$
Using (II) in Definition \ref{definition:continuous Hilbert bundle}, choose
$\omega_1,\dots,\omega_{n_{s_0}}\in \Omega$ such that
$\omega_j(s_0)=\xi_j$, $1\le j\le n_{s_0}$, and let
$\kappa_{s_0}=\sum_{j=1}^{n_{s_0}}\Theta_{\omega_j,\omega_j}$.
By continuity of the operator fields in $A$, there is an open set
$U_{s_0}\subset\Delta$ containing $s_0$ such that
$\|h(s)-\kappa_{s_0}(s)\|<\varepsilon$ for all $s\in U_{s_0}$.

This procedure leads to an open cover $\{U_s\}_{s\in\Delta}$ of $\Delta$, from which
(by compactness) there exists a finite subcover
$\{U_1,\dots,U_m\}$ and corresponding fields
$\kappa_i=\sum_{j=1}^{n_{i}}\Theta_{\omega_j^{[i]},\omega_j^{[i]}}$.
Let $\{\psi_1,\dots,\psi_m\}\subset C(\Delta)$
be a partition of unity subordinate to $\{U_1,\dots,U_m\}$
and note that
$\psi_i\cdot\Theta_{\omega_j^{[i]},\omega_j^{[i]}}\,
=\,\Theta_{\psi_i^{1/2}\cdot\omega_j^{[i]},\psi_i^{1/2}\cdot\omega_j^{[i]}}$
for all $j$ and $i$. Hence, the field
$\kappa=\sum_{i=1}^m\psi_i\cdot\kappa_i$ is in $F_+$, and for each $s\in \Delta$,
\[
\|h(s)-\kappa(s)\| = \|\displaystyle\sum_{i=1}^m\psi_i\cdot(h-\kappa_i)(s) \|
\leq \displaystyle\sum_{i=1}^m\psi_i(s)\|(h-\kappa_i)(s) \|
<\varepsilon\,.
\]
Hence, $h$ is in the norm-closure of $F_+$.
\end{proof}

\begin{lemma}\label{lemma:pegado1}
Let $\{U_i\}_{i\in\Lambda}$ be a family of pairwise disjoint
clopen subsets of $\Delta$ whose union $U$ is dense in $\Delta$,
and let $c_i=\chi^{\phantom{a}}_{U_i}\in C(\Delta)$, for each $i\in\Lambda$.
Suppose that $\{\omega_i\}_{i\in \Lambda}$ is any bounded family
in $\Omega$ and let $\tilde \omega=\sum_{i\in \Lambda} c_i\,
\omega_i\ \in\Omegawk$, in the sense of Remark \ref{remark:about the sum notation}.
If $f\in C(\Delta)$ is such that $f(s)=0$
for $s\in \Delta\setminus U$, then $f\cdot \tilde \omega\in
\Omega$.
\end{lemma}

\begin{proof}
Fix $s_0\in \Delta$ and let $\varepsilon > 0$. If $s_0\in
\Delta\setminus U$, then by the continuity of $f$ and the fact
that $f(s_0)=0$ there exists an open subset $U_{s_0}\subset\Delta$
containing $s_0$ such that
$|f(s)|<\varepsilon\|\tilde\omega\|^{-1}$ for all $s\in U_{s_0}$.
Hence, the vector field $f\cdot\tilde\omega$ is within
$\varepsilon$ of the zero vector field $0\in\Omega$ on the open
set $U_{s_0}$.

On the other hand, if $s_0\in U$, then there exists $j\in \Lambda$
such that $s_0\in U_j$. By construction, $c_j\cdot \tilde
\omega=c_j \cdot\omega_j$ and so $\tilde \omega(s)=\omega_j(s)$
for all $s\in U_j$. Because
 $\|(f\cdot \tilde \omega)(s)-(f\cdot\omega_j)(s)\|=0$ for all $s\in U_j$,
the vector field $f\cdot\tilde\omega$ is within $\varepsilon$ of the vector field
$f\cdot\omega_j\in\Omega$ on the open set $U_{j}$.
Thus, by the local uniform approximation property (axiom (IV) in Definition
\ref{definition:continuous Hilbert bundle}), $f\cdot\tilde\omega\in\Omega$.
\end{proof}

The fact that $\Omega^\perp=\{0\}$ in $\Omegawk$ (Lemma \ref{lemma:density})
suggests that $\Omega$ is somehow dense in $\Omegawk$. The
next proposition makes this relation more explicit.

\begin{proposition} \label{proposition:pegado2} If $\nu\in \Omegawk$ and
$\varepsilon>0$, then there exist a family $\{c_i\}_{i\in
\Lambda}$ of pairwise orthogonal projections in $C(\Delta)$ with
supremum $1$ and a bounded family $\{\omega_i\}_{i\in
\Lambda}\subset\Omega$ such that $\|\nu-\sum_{i\in
\Lambda}c_i\cdot \omega_i\|<\varepsilon$.
\end{proposition}

\begin{proof} By Lemma \ref{lemma:lowersemicontinuous}, the function $s\mapsto \|\nu(s)\|$ is
lower semicontinuous; hence, there exists a meagre set $M_{\nu}$
such that the function $s\mapsto \|\nu(s)\|$ is continuous in the
relative topology of $\Delta\setminus M_\nu$. Observe that
$\overline{(\Delta\setminus M_\nu)}=\Delta$.

Fix $s_0\in \Delta\setminus M_\nu$ and let $\omega\in\Omega $ be such that $\omega(s_0)=\nu(s_0)$.
Since
\[
\|\nu(s)-\omega(s)\|^2\,=\,\|\nu(s)\|^2+\|\omega(s)\|^2-2 \re \langle \nu,\omega\rangle(s) \,,
\]
the continuity in the relative topology of $\Delta\setminus M_\nu$
guarantees the existence of an open subset $U_{s_0}$ of $\Delta$
containing $s_0$ such that $\|\nu(s)-\omega(s)\|< \varepsilon/2$
for all $s\in (\Delta\setminus M_\nu)\cap U_{s_0}$. Hence, again
by continuity we get that $\|\nu-\omega\|(s)< \varepsilon$ for
all $s\in \overline U_{s_0}$. The set $\overline U_{s_0}$ is a
clopen subset of $\Delta$ and $\Delta'=\Delta\setminus \overline
U_{s_0}$ is also a Stonean space. Further,
$M_\nu\cap\Delta'=M_\nu\cap (\Delta\setminus \overline U_{s_0})$ is a
meagre set such that the function $s\mapsto \|\nu(s)\|$, for $s\in
\Delta'\setminus (M_\nu\cap\Delta')$, is continuous in the
relative topology.

An application of Zorn's Lemma yields a maximal family
$\{(\chi^{\phantom{a}}_{U_i},\omega_i)\}_{i\in \Lambda}$ such that $U_i\cap
U_j=\emptyset$ for $i\neq j$ and such that
$\|\chi^{\phantom{a}}_{U_i}(\nu-\omega_i)\|< \varepsilon$. Maximality ensures
that $\overline {(\cup_{i\in I}U_i)}=\Delta$, for otherwise we
can enlarge this family by the previous procedure in the Stonean
space $\Delta\setminus\overline {(\cup_{i\in \Lambda}U_i)}$. If we let $c_i=\chi_{U_i}$ for $i\in \Lambda$ then
it is clear by Lemma \ref{lemma:lowersemicontinuous} that $\|\nu-\sum_{i\in
\Lambda}c_i\cdot \omega_i\|<\varepsilon$ as for every $j\in \Lambda$ we have that  $\|c_j(\nu-\sum_{i\in
\Lambda}c_i\cdot \omega_i)\|=\|c_j(\nu-\omega_j)\|< \varepsilon$ and $\bigvee_{i\in \Lambda}\,c_i=1$.\end{proof}

The next result is the key step in the proof of Theorem \ref{theorem:extension of Somerset}.

\begin{proposition}\label{proposition:abelian projections} For every abelian projection
$e\in B(\Omega_{\rm wk})$ and $\varepsilon>0$
there is an essential ideal $I\subset K(\Omega)$ and $x\in M(I)$ such that $\|e-x\|<\varepsilon$.
\end{proposition}
\begin{proof}
Assume that $e\in B(\Omega_{\rm wk})$ is an abelian projection and
let $\varepsilon>0$. Thus, by \cite[Lemma 13]{kaplansky1953},
$e=\Theta_{\nu,\nu}$ for some $\nu\in\Omega_{\rm wk}$ for which
$\langle\nu,\nu\rangle$ is a projection of $C(\Delta)$. By
Proposition \ref{proposition:pegado2}, there is a family
$\{c_i\}_{i\in \Lambda}$ of pairwise orthogonal projections in
$C(\Delta)$ with supremum $1$ and a bounded family
$\{\omega_j\}_{j\in \Lambda}\subset\Omega$ such that
$\|\nu-\tilde\omega\|<\varepsilon/(2\|\nu\|)$, where
$\tilde\omega=\sum_{j\in \Lambda}c_j\cdot \omega_j\in\Omega_{\rm
wk}$. Each $c_j$ is the characteristic function of a clopen set
$U_j$ and the union $U$ of these sets $U_j$ is dense in $\Delta$.

Let $I=\{a\in K(\Omega)\,:\,a(s)=0,\,\forall\,s\in\Delta\setminus U\}$, which is an
essential ideal of $K(\Omega)$.
Define $F^I\subset F_+\subset K(\Omega)_+$ to be the set
\[
F^I \,=\,
\{\displaystyle\sum_{i=1}^n\Theta_{\mu_i,\mu_i}\,:\,
n\in\mathbb N,\,\mu_i\in\Omega,\,\mu_i|_{\,\Delta\setminus U}=0,\ i=1,\ldots,n\}\,.
\]
Suppose that $\eta\in\Omega$ satisfies $\|\eta(s)\|=0$ for all $s\in\Delta\setminus U$,
and consider $\Theta_{\eta,\eta}\in F^I$.
Observe that
$\Theta_{\tilde \omega, \tilde \omega}\,\Theta_{\eta,\eta}\,=\,\Theta_{\langle\eta,\tilde\omega\rangle\cdot\tilde\omega,\,\eta}$, which
is an element of
$I$
because $\langle\eta,\tilde\omega\rangle(s)=\langle\eta(s),\tilde\omega(s)\rangle=0$ for all $s\in\Delta\setminus U$ and
$\langle\eta,\tilde\omega\rangle\cdot\tilde\omega\in\Omega$ by Lemma \ref{lemma:pegado1}.
Hence, $\Theta_{\tilde \omega, \tilde \omega}$ maps the set $F^I$ back into $I$.
Because $F^I$ is dense in $I_+$, as we shall show below,
$\Theta_{\tilde \omega, \tilde \omega}I\subset I$ and a similar computation shows that
$I\Theta_{\tilde \omega, \tilde \omega}\subset I$.
Furthermore, writing $x=\Theta_{\tilde \omega, \tilde \omega}$,
\[
\|e-x\|=\|\Theta_{\nu,\nu}-\Theta_{\tilde \omega, \tilde \omega}\|
\leq (\|\nu\|+\|\tilde \omega\|)\, \|\nu-\tilde \omega\|
<\varepsilon.
\]

It remains to show that $F^I$ is dense in $I_+$. To this end,
assume $\varepsilon'>0$ and $\kappa\in I_+$. Thus, $\kappa(s)=0$
for all $s\in\Delta\setminus U$. Furthermore, by Lemma
\ref{lemma:pos dense}, there exists $h\in F_+$ such that
$\|\kappa-h\|<\varepsilon'$. Let $\tilde h=\chi^{\phantom{a}}_{\Delta\setminus
U}\cdot h$ and note that, as $\kappa\in I$, it is also true that
$\|\kappa-\tilde h\|<\varepsilon'$. Now if $h$ has the form
$\sum_{j=1}^n\Theta_{\mu_j,\mu_j}$ for some
$\mu_j\in\Omega$, then $\tilde
h=\sum_{j=1}^n\Theta_{\chi^{\phantom{a}}_{\Delta\setminus
U}\mu_j,\chi^{\phantom{a}}_{\Delta\setminus U}\mu_j}\in F^I$.
\end{proof}

\begin{proof}[Proof of Theorem \ref{theorem:extension of Somerset}]
Because $K(\Omega)$ is an ideal of $A$, we have $M(A)\subset M(K(\Omega))$. Moreover, as
$K(\Omega)$ is an essential ideal of $A$ we conclude that $\mloc(A)=\mloc(K(\Omega))$ \cite[Proposition 2.3.6]{Ara--Mathieu-book}.
On the other hand, the inclusions
\[
B(\Omega)=M(K(\Omega)) \,\subset\,\mloc(K(\Omega)) \,\subset\,
\mloc\left(\mloc(K(\Omega))\right)\subset B(\Omegawk)
\]
hold by \cite[Theorem 4.6]{frank--paulsen2003}.

Therefore, we are left to show that $\mloc\left(\mloc(K(\Omega))\right)=B(\Omega_{\rm wk})$.
By \cite[Corollary 4.3]{frank--paulsen2003},
an element $z\in I(K(\Omega))=B(\Omegawk)$ belongs to $\mloc(K(\Omega))$
if and only if for every $\varepsilon>0$ there is an essential
ideal $I\subset K(\Omega)$ and a multiplier $x\in M(I)$ such
that $\|z-x\|<\varepsilon$.
By Proposition \ref{proposition:k-omegawk}, $K(\Omegawk)$ is the
(essential) ideal
of $B(\Omega_{\rm wk})$ generated by the abelian projections
of $B(\Omega_{\rm wk})$; thus, by Proposition \ref{proposition:abelian projections},
$K(\Omegawk)\subset \mloc(K(\Omega))$. Hence, $K(\Omegawk)$ is an essential
ideal of $\mloc(K(\Omega))$ and so $M(K(\Omegawk))\subset \mloc\left(\mloc(K(\Omega))\right)$. However, $B(\Omega_{\rm wk})=M(K(\Omegawk))$ by
Kasparov's Theorem
\cite[Theorem 2.4]{Lance-book} (or by a theorem of Pedersen \cite{pedersen1984}); hence,
\[
B(\Omega_{\rm wk})=M\left( K(\Omegawk)\right)\,\subset \mloc\left(\mloc(K(\Omega))\right)\,\subset\,B(\Omega_{\rm
wk})\,,
\]
which yields $\mloc\left(\mloc(K(\Omega))\right)=B(\Omega_{\rm wk})$.
\end{proof}

Somerset has shown that every separable postliminal (that is, type I) C$^*$-algebra $A$ has the property that
$\mloc(\mloc (A))=I(A)$ \cite[Theorem 2.8]{somerset2000}. Theorem \ref{theorem:extension of Somerset}
demonstrates that the same behavior occurs with (certain) nonseparable type I C$^*$-algebras. Somerset's methods
are different from ours in at least two ways: he employs the Baire $*$-envelope of a C$^*$-algebra where we use the injective
envelope
and he uses properties of Polish spaces---spaces that arise from the separability of the algebras under study.
It is reasonable to conjecture that $\mloc(\mloc (A))=I(A)$ for all C$^*$-algebras $A$ that possess a
postliminal essential ideal. To prove such a statement, it would be enough to prove it for any continuous trace
C$^*$-algebra $A$.

%%%%%%%%%%%%%%%%%%%%%%%%%%%%%%%%%%%%%%%%%%%%%%
\section{Direct Sum Decompositions}\label{section:reduction}

A Kaplansky--Hilbert module $E$ over $C(\Delta)$ is said to be
\emph{homogeneous} \cite{kaplansky1953} if there is a subset $\{\nu_j\}_{j\in\Lambda}\subset E$
-- called an \emph{orthonormal basis} --
such that $\langle\nu_i,\nu_j\rangle=0$ for all $j\not= i$, $|\nu_j|=1$ for all $j$, and $\{\nu_j\}_{j\in\Lambda}^\bot=\{0\}$, where
for any $\nu\in E$, $|\nu|$ is the continuous real-valued function $|\nu|=\langle\nu,\nu\rangle^{1/2}\in C(\Delta)$.

Kaplansky introduced the notion of homogeneous AW$^*$-module with
the aim of reducing the study of abstract AW$^*$-modules to the
slightly more concrete setting in which the modules have an
orthonormal basis. This is justified by the following
result:
\begin{theorem}[\cite{kaplansky1953}]\label{theorem:theorem 1 in Kaplansky1953}
Let $E$ be a Kaplansky-Hilbert module over $C(\Delta)$. Then there
exist orthogonal projections $\{c_i\}_{i\in I}\subset C(\Delta)$
with supremum 1 such that $c_i\,E$ is a homogenous AW$^*$-module
over $c_i\,C(\Delta)$.
\end{theorem}

Note that in the situation of Theorem \ref{theorem:theorem 1 in Kaplansky1953},
for each $i$ there exists a clopen set $\Delta_i\subset\Delta$ with $c_i=\chi_{\Delta_i}$.
The sets $\{\Delta_i\}$ are pairwise disjoint, and $\cup_i\Delta_i$ is dense in $\Delta$.

 In this section we consider the effect of a direct sum decomposition in the structures
 that have been studied in the previous sections, namely the Fell algebra $A$ of the weakly
 continuous Hilbert bundle $(\Delta,\{H_s\}_{s\in\Delta},\Omega,\Omegawk)$, and its local
 multiplier algebra $\mloc(A)$.
 We show that a decomposition of
$\Omegawk$ into a direct sum $\oplus_i c_i\Omegawk$ given by a partition
of the identity $\{c_i\}$ in $C(\Delta)$ leads one to consider two corresponding direct sum
C$^*$-algebras: $\oplus_iA_i$ and $\oplus_i\mloc(A_i)$, where
$A_i$ is a subalgebra of $A$ for all $i$. We prove that $A$ need
not be isomorphic to $\oplus_iA_i$, yet
$\mloc(A)\cong\oplus_i\mloc(A_i)$. The latter result is especially
interesting if one recalls that $\mloc(A)$ is generally not an
AW$^*$-algebra \cite[Theorem 6.13]{ara-mathieu2008}.

\begin{theorem}\label{reduction theorem} Let $(\Delta,\{H_s\}_{s\in \Delta},\Omega)$
be a continuous Hilbert bundle over the Stonean space $\Delta$.
Assume that $\{\Delta_i\}_{i\in I}$ is a family of
pairwise-disjoint clopen subsets of $\Delta$ whose union is dense
in $\Delta$, and for each $i\in I$ let $c_i=\chi_{\Delta_i}\in
C(\Delta)$ and
$\Omega_i=\{\omega_{\vert\Delta_i}\,:\,\omega\in\Omega\}$. Then:
\begin{enumerate}
\item\label{reduction theorem:1} $(\Delta_i, \{H_s\}_{s\in\Delta_i}, \Omega_i)$ is a continuous Hilbert bundle;
\item\label{reduction theorem:2} $(\Omega_i)_{\rm wk}\cong\,c_i\cdot\Omegawk$ as C$^*$-modules;
\item\label{reduction theorem:3} $\Omegawk\cong\bigoplus_i(\Omega_i)_{\rm wk}$ as C$^*$-modules;
\item\label{reduction theorem:4} $B((\Omega_i)_{\rm wk})\cong\,c_i\cdot B(\Omegawk)$ as C$^*$-algebras;
\item\label{reduction theorem:5} $B(\Omegawk)\cong\bigoplus_iB((\Omega_i)_{\rm wk})$ as C$^*$-algebras.
\end{enumerate}
In \ref{reduction theorem:2} and \ref{reduction theorem:3}, the isomorphism is considered together
with the identification $C(\Delta_i)\simeq c_i\,C(\Delta)$.
\end{theorem}

\begin{proof}
Being clopen in $\Delta$, each $\Delta_i$ is itself a Stonean
space, and it is easy to see that $C(\Delta_i)\cong c_i\,C(\Delta)$

\noindent\textit{\ref{reduction theorem:1}}. For axiom (I) in Definition \ref{definition:continuous Hilbert bundle}, we aim to show that $\Omega_i$ is a $C(\Delta_i)$ module. Let $\omega\in\Omega$ and consider
$\omega_i=\omega|_{\Delta_i}$. Choose any $f_i\in C(\Delta_i)$. As $\Delta_i$ is clopen,
$f_i$ can be extended to $F_i\in C(\Delta)$
such that $f_i=F_i|_{\Delta_i}$, and $F_i|_{\Delta\setminus\Delta_i}=0$.
The action $f_i\cdot\omega_i=(F_i\cdot\omega)|_{\Delta_i}$ gives
$\Omega_i$ the structure of a $C(\Delta_i)$ module. Axioms (II) and (III) of Definition \ref{definition:continuous Hilbert bundle} are trivially satisfied.

For axiom (IV), let $\xi:\Delta_i\rightarrow\bigsqcup_{s\in\Delta_i}\,H_s$ be a
vector field
such that for every $s_0\in \Delta_i$ and $\varepsilon>0$ there is an open set $U_i\subset \Delta_i$
containing $s_0$ and a $\omega_i\in \Omega_i$ with $\|\omega_i(s)-\xi(s)\|<\varepsilon$
for all $s\in U_i$. Let
$\Xi: \Delta_i\rightarrow\bigsqcup_{s\in\Delta}\,H_s$ be the vector field that coincides with $\xi$ on $\Delta_i$
and is identically zero off $\Delta_i$.
By the definition of $\Omega_i$,
there is $\omega\in\Omega$ such that $\omega_i=\omega|_{\Delta_i}$.
The set $U_i$ is also open in $\Delta$, and
$\|\omega(s)-\Xi(s)\|<\varepsilon$ for all $s\in U_i$.
If $s_0\not\in\Delta_i$ choose any
open set $V_i$ containing $s_0$ such that $V_i\cap U_i=\emptyset$ and let $\omega\in\Omega$ be arbitrary; then
$0=\|\chi_{\Delta_i}(s)\omega(s)-\Xi(s)\|<\varepsilon$ for all $s\in V_i$. Since $\chi_{\Delta_i}\cdot\omega\in\Omega$
and since $\Omega$ is closed under local uniform approximation,
$\Xi\in\Omega$, whence $\xi\in\Omega_i$.

\noindent\textit{\ref{reduction theorem:2}}. Let
$T_i:c_i\,\Omega_{\rm wk}\rightarrow (\Omega_i)_{\rm wk}$
be given by $T_i(c_i\nu)=\nu|_{\Delta_i}$. It is clear that
$T_i$ is well defined, linear, bounded, and has trivial kernel; to show that it is onto,
note that if $\nu_i\in (\Omega_i)_{\rm wk}$,
then---since $\Delta_i$ is clopen---the
vector field $\nu:\Delta\rightarrow\bigsqcup_{s\in\Delta}\,H_s$ defined by $\nu(s)=0$, for $s\not\in\Delta_i$, and
$\nu(s)=\nu_i(s)$, for $s\in \Delta_i$, has the property that $\langle\omega,\nu\rangle\in C(\Delta)$,
for all $\omega\in\Omega$; so $\nu\in\Omegawk$  and $\nu_i=T_i(c_i\nu)$. It is also easy to check
that $T_i$ preserves inner products.

\noindent\textit{\ref{reduction theorem:3}}. Let $T:\Omegawk\rightarrow\bigoplus_i(\Omega_i)_{\rm wk}$,
given by $T\nu=\left( T_i(c_i\nu)\right)_{i\in I}$.
The previous paragraph and Lemma \ref{lemma:norm off a meagre set}
show that $T$ is an isometry; we show now that $T$ is onto. Suppose that
$\nu'=(\nu_i)_{i\in I}\in\bigoplus_i(\Omega_i)_{\rm wk}$.
For each $i\in I$ let $\tilde\nu_i$ denote the vector field on $\Delta$ that coincides
with $\nu_i$ on $\Delta_i$ and vanishes elsewhere.
Then $\tilde\nu_i\in\Omegawk$ and $T_i(c_i\tilde\nu_i)=\nu_i$. Hence, if
$\nu=\sum_ic_i\tilde\nu_i$ as in Remark \ref{remark:about the sum notation}, we have $\nu\in\Omegawk$ and $T\nu=\nu'$. Thus, $\Omegawk$ and
$\bigoplus_i(\Omega_i)_{\rm wk}$ are isomorphic Banach spaces.
Similar arguments show that $\bigoplus_i(\Omega_i)_{\rm wk}$
is a $C(\Delta)$-module and that $T$ is module isomorphism. Hence, $\Omegawk\cong\bigoplus_i(\Omega_i)_{\rm wk}$ as C$^*$-modules.

\noindent\textit{\ref{reduction theorem:4}}. Let $\rho_i:c_i\,B(\Omegawk)\rightarrow B((\Omega_i)_{\rm wk})$ be given by
$\rho_i(c_ib)\,T_i(c_i\nu)=(b\nu)|_{\Delta_i}$. This map is well-defined because if
$c_ib_1=c_ib_2$ then for any $\nu\in\Omegawk$ we have $(b_1\nu)|_{\Delta_i}
=(c_ib_1\nu)|_{\Delta_i}=(c_ib_2\nu)|_{\Delta_i}=(b_2\nu)|_{\Delta_i}$.
A similar computation shows that $\rho_i$ is one-to-one, and linearity is clear. To see
that $\rho_i$ is onto, let $b_i\in B((\Omega_i)_{\rm wk})$. Consider the injection
$\tilde{}:(\Omega_i)_{\rm wk}\rightarrow \Omegawk$ where $\tilde{\nu_i}\in\Omegawk$
is the vector field that agrees with $\nu_i$ on $\Delta_i$ and is 0 elsewhere. Let
$b\in B(\Omegawk)$ be the operator given by $b\nu=\widetilde{b_i(\nu|_{\Delta_i})}$.
Then $\rho_i(c_ib)(T_ic_i\nu)=(b\nu)|_{\Delta_i}=\widetilde{b_i(\nu|_{\Delta_i})}|_{\Delta_i}
=b_i(\nu|_{\Delta_i})=b_i\,(T_ic_i\nu)$, so $\rho_i(c_ib)=b_i$.

\noindent\textit{\ref{reduction theorem:5}}. Let
$\rho: B(\Omegawk)\rightarrow\bigoplus_i B((\Omega_i)_{\rm wk})$ be the map
$\rho(b)=(\rho_i(c_ib))_{i\in I}\,.$
It is clear that $\rho$ is a homomorphism. If $\rho(b)=0$ for some $b\in B(\Omegawk)$, then
-- as each $\rho_i$ is one-to-one -- $c_ib=0$ for all $i$; this implies that
$b^*b=b^*(\sup_i(c_i\cdot I))b=\sup_i(b^*c_ib)=0$ by \cite[Corollary 4.10]{hamana1981},
so $b=0$ and $\rho$ is
one-to-one. To show that $\rho$ is onto, let $(b_i)_i\in \bigoplus_i B((\Omega_i)_{\rm wk})$;
as each $\rho_i$ is onto, there exist operators $b^i\in B(\Omegawk)$
with $\rho_i(c_ib^i)=b_i$. Define $b\in B(\Omegawk)$ by $b\nu=\sum_i\,c_ib^i\nu$
(in the sense of Remark \ref{remark:about the sum notation}; that is, $c_ib\nu=c_ib^i\nu$). Then
$\rho_i(c_ib)\nu|_{\Delta_i}=(c_ib\nu)|_{\Delta_i}=(c_ib^i\nu)|_{\Delta_i}
=\rho_i(c_ib^i)\nu|_{\Delta_i}=b_i\nu|_{\Delta_i}$. So $\rho(b)=(b_i)_i$.
\end{proof}

\begin{proposition}\label{proposition:examples 1}  Assume the notation, hypotheses, and
conclusions of Theorem \ref{reduction theorem}. Then
there exists an example where the canonical embedding
$\Omega\hookrightarrow\bigoplus_i\Omega_i$ (via the isometry $T$ from the proof of
\ref{reduction theorem:3} in Theorem \ref{reduction theorem})
is not onto. In particular, $\Omega$ is properly contained in $\Omegawk$.
\end{proposition}

\begin{proof} Take $\Delta$ and the family of clopen subsets $\{\Delta_i\}_{i\in I}$
in Theorem \ref{reduction theorem} to be such that
$\bigcup_{i\in I}\Delta_i\neq \Delta$. Thus,
$I$ is an infinite set. Let $H$ be a Hilbert space with orthonormal basis $\{e_i\}_{i\in I}$ and consider the trivial Hilbert bundle
$\Omega=C(\Delta,H)$ of all continuous functions $\omega:\Delta\rightarrow H$. As in Theorem \ref{reduction theorem}, let
$\Omega_i=C(\Delta_i,H)$.

For each $i\in I$, set $\omega_i\in \Omega$ with $\omega_i(s)=e_i$ for all $s$
and consider $(\omega_i)_{i\in I}\in\bigoplus_i\Omega_i$.
Under the isomorphism of Theorem \ref{reduction theorem}, this
element $(\omega_i)_{i\in I}$ is identified with
$\omega=\sum_{i\in I} \chi_{\Delta_i}\cdot \tilde\omega_i\in \Omegawk$ (in the sense of Remark \ref{remark:about the sum notation}), where $\tilde\omega_i$ is any
element of $\Omega$ that agrees with $\omega_i$ on $\Delta_i$ and vanishes off $\Delta_i$.
Under this identification,
$\omega\notin\Omega$; that is, the function $s\mapsto\|\omega(s)\|$ fails to be continuous on $\Delta$.
We argue this by contradiction.

Assume that $s\mapsto\|\omega(s)\|$ is continuous on $\Delta$. Because
$\|\omega(s)\|=1$ for all $s\in \cup_{i\in I} \Delta_i$, continuity implies that $\|\omega(s)\|=1$ for $s\in \Delta$.
Choose $s_0\in \Delta\setminus (\cup _{i\in I} \Delta_i)$ and let $(s_\alpha)_{\alpha\in \Lambda}\subset\cup _{i\in I} \Delta_i$ be a net such that $s_\alpha\rightarrow s_0$.
Let $\eta\in \Omega$ be the constant field $\eta(s)=\omega(s_0)$, for all $s\in \Delta$.
Since $\omega\in \Omegawk$, we have
\begin{equation}\label{equation:examples-1:1}
\lim_\alpha\ \langle \omega(s_\alpha),\eta(s_\alpha)\rangle\,=\,\langle \omega(s_0),\eta(s_0)\rangle\,=\,\langle \omega(s_0),\omega(s_0)\rangle\,=\,1\,.
\end{equation}
For each $\alpha\in \Lambda$ let $i(\alpha)\in I$ be such that $s_\alpha\in \Delta_{i(\alpha)}$. Thus,
for every $\alpha\in \Lambda$, $I_\alpha=\{i(\beta):\ \beta\in I,\ \beta\geq \alpha\}$ is an infinite set (for otherwise $s_0\in \Delta_i$ for some $i\in I$).
Therefore,
\begin{equation}\label{equation:examples-1:2}
\lim_\alpha\ \langle \omega(s_\alpha),\eta(s_\alpha)\rangle\,=\,\lim_\alpha\ \langle e_{i(\alpha)},\omega(s_0)\rangle\,=\,0\,.
\end{equation}
As \eqref{equation:examples-1:1} and \eqref{equation:examples-1:2} cannot be true simultaneously, we obtain a contradiction. Hence, $\omega\notin \Omega$.
\end{proof}

Our second reduction theorem below notes some consequences of Theorem \ref{reduction theorem}
when applied to the injective envelope and local multiplier algebras of the Fell algebra $A$
associated to a continuous Hilbert bundle.

\begin{theorem}\label{theorem:reduction theorem-2} Let $(\Delta,\{H_t\}_{t\in \Delta},\Omega)$
be a continuous Hilbert bundle over the Stonean space $\Delta$ and let $A=(\Delta,\{K(H_t\},\Gamma)$
denote the associated
continuous trace C$^*$-algebra of Fell.
Assume that $\{\Delta_i\}_{i\in I}$ is a family of
pairwise-disjoint clopen subsets of $\Delta$ whose union is dense
in $\Delta$, and for each $i\in I$ let $c_i=\chi_{\Delta_i}\in
C(\Delta)$ and
$\Omega_i=\{\omega_{\vert\Delta_i}\,:\,\omega\in\Omega\}$. Then:
\begin{enumerate}
\item\label{theorem:reduction theorem-2:1} if $A_i$ denotes the Fell algebra of $(\Delta_i, \{H_s\}_{s\in\Delta_i}, \Omega_i)$, then $A_i\cong c_i\cdot A$;
\item\label{theorem:reduction theorem-2:2} $I(A_i)=B((\Omega_i)_{\rm wk})$;
\item\label{theorem:reduction theorem-2:3} $I(A)\cong\bigoplus_{i\in I}I(A_i)$;
\item\label{theorem:reduction theorem-2:4} $\mloc(A)\cong\bigoplus_{i\in I} \mloc(A_i)$.
\end{enumerate}
\end{theorem}

\begin{proof}
Let $A_i=(\Delta_i,\{K(H_s)\}_{s\in\Delta},\Gamma_i)$ denote the Fell C$^*$-algebra
associated to the Hilbert bundle
$(\Delta_i, \{H_s\}_{s\in\Delta_i}, \Omega_i)$. That is, $\Gamma_i$
consists of all weakly continuous almost finite-dimensional operator fields
$a_i:\Delta_i\rightarrow \bigsqcup_{s\in\Delta_i}\,K(H_s)$ such that $s\mapsto\|a_i(s)\|$
is continuous. We have that $B((\Omega_i)_{\rm wk})$ is a type I AW$^*$-algebra
with centre $C(\Delta_i)$.

\noindent\textit{\ref{theorem:reduction theorem-2:1}}. For each $a_i\in\Gamma_i$ there is an $a\in\Gamma$ such that $a_i=a|_{\Delta_i}$.
To verify this, let $a:\Delta_i\rightarrow \bigsqcup_{s\in\Delta}\,K(H_s)$ be the operator field defined by
$a(s)=a_i(s)$, for $s\in\Delta_i$, and $a(s)=0$, for $s\not\in\Delta_i$. Since $\Delta_i$ is a clopen set,
the maps $s\rightarrow\|a(s)\|$ and $s\mapsto\langle a(s)\omega_1(s),\omega_2(s)\rangle$ are continuous for every $\omega_1,\omega_2\in\Omega$.
The operator field $a$ is also locally finite-dimensional, again because $\Delta_i$
is clopen and $a_i$ has the property on $\Delta_i$.
Hence, $a\in\Gamma$. Next, let $\pi_i:A_i\rightarrow c_iA$ be defined by $\pi_i(a_i)=c_ia$, where $a\in A$ is
any operator field that restricts to $a_i$ on
$\Delta_i$. This map is clearly well-defined, and a homomorphism.

\noindent\textit{\ref{theorem:reduction theorem-2:2}}. By Theorem \ref{theorem:inclusion}, $B((\Omega_i)_{\rm wk})=I(A_i)=I(c_iA)$.

\noindent\textit{\ref{theorem:reduction theorem-2:3}}.  By \cite[Lemma 6.2]{hamana1981}, $I(c_iA)=c_iI(A)$.
Hence, $I(A_i)=B((\Omega_i)_{\rm wk})$ and Theorem \ref{reduction theorem} immediately yields
$I(A)\cong\bigoplus_{i\in I}I(A_i)$.

\noindent\textit{\ref{theorem:reduction theorem-2:4}}. We take each $\mloc(A_i)$ to be a
C$^*$-subalgebra of $B((\Omega_i)_{\rm wk})$.
First we remark that the isomorphism $\rho$ from Theorem \ref{reduction theorem}
sends $A$ into $\bigoplus_iA_i$. To see why, recall that
$a\nu(s)=a(s)\nu(s)$, for all $a\in A$, $\nu\in\Omegawk$, and $s\in\Delta$
(Proposition \ref{proposition:b(Omega) in B(Omegawk)}). Since, for a given $i\in I$,
the action of $\rho_i(a)$ on $\nu_i\in(\Omega_i)_{\rm wk}$ is defined by $\nu_i\mapsto (a\nu)|_{\Delta_i}$,
where $\nu\in\Omegawk$ is any vector with $\nu|_{\Delta_i}=\nu_i$,
it is easy to verify that $\rho_i(a)$
is a weakly continuous almost finite-dimensional operator field on $\Delta_i$.

To show that $\rho\left(\mloc(A)\right)\subset\bigoplus_i\mloc(A_i)$,
let $x\in\mloc(A)\subset I(A)$ and suppose that $\varepsilon>0$.
 Thus, there is an essential ideal
$J\subset A$ and a multiplier $x\in M(J)$ such that $\|x-y\|<\varepsilon$.
Further, there exists an open dense subset $U\subset \Delta$
such that
\begin{equation}\label{eq J is ideal}
J\,=\,\{a\in A: a(s)=0\, , \  s\in \Delta\setminus U\}\,.
\end{equation}
For $i\in I$, let $U_i=\Delta_i \cap U$, which is an open dense set in $\Delta_i$. Therefore,
\begin{equation}\label{eq Ji is ideal}
J_i\,=\,\{a_i\in A_i:\,  a(s)=0,\ s\in \Delta_i\setminus U_i \}
\end{equation}
is an essential ideal in $A_i$. We aim to show that $\rho_i(y)\in M(J_i)$. To this end, select $a_i\in J_i$.
As $A_i\cong c_i\cdot A$, there is an $a\in A$ such that $a_i(s)=a(s)$ for
all $s\in\Delta_i$.
 Moreover, $a\in A$ can be chosen
so that $a(s)=0$ for all $s\in\Delta\setminus\Delta_i$.

Because $a_i\in J_i$, we conclude that $a(s)=0$ for all $s\in\Delta\setminus U$;
that is, $a\in J$. Therefore, $ya\in J$, which implies that
$ya(s)=0$ for all $s\in \Delta\setminus U$.
In particular, $ya(s)=0$ for all $s\in \Delta_i\setminus U_i$.
The element $\rho_i(y)a_i\in B((\Omega_i)_{\rm wk})$ is in fact an
operator field since $\rho_i(y)a_i=\rho_i(y)\rho_i(c_ia)=\rho_i(c_i(ya))\in A_i$.
Then, for all $s\in \Delta_i\setminus U_i$ and $\nu\in\Omegawk$,
\begin{align*}
[\rho_i(y)a_i](s)(T_ic_i\nu)(s)&=\rho_i(y)a_i(T_ic_i\nu)(s)
=\rho_i(c_iya)(T_ic_i\nu)(s)\\
&=(ya)\nu|_{\Delta_i}(s)=(ya)(s)\nu|_{\Delta_i}(s)=0.
\end{align*}
With $\nu$ being arbitrary, we conclude that $\rho_i(y)a_i(s)=0$, that is
$\rho_i(y)a_i\in J_i$,
and so $\rho_i(y)$ is a left multiplier of $J_i$. By a similar argument,
$\rho_i(y)$ is a right multiplier of $J_i$, and so $\rho_i(y)\in M(J_i)$.
 Thus, $\rho(y)\in\bigoplus_i\mloc(A_i)$
and
$\|\rho(x)-\rho(y)\|=\|x-y\|<\varepsilon$. As $\varepsilon>0$ was chosen arbitrarily, this proves that
$\rho(x)\in\bigoplus_i \mloc(A_i)$.

Conversely, let us show that
$\bigoplus_i \mloc(A_i)\subset\rho(\mloc(A))$.
Let $(x_i)_i\in \bigoplus_i \mloc(A_i)$; thus for each $i\in I$,
there exist an essential ideal $J_i\subset A_i$ and  $y_i\in M(J_i)$ such that $\|x_i-y_i\|< \varepsilon$ for all $i\in I$. For each $i\in I$,
there exists an open dense subset $U_i\subset \Delta_i$ such that $J_i$ is given as in \eqref{eq Ji is ideal}. Define $U=\bigcup_{i\in I} U_i$,
which is an open dense subset of $\Delta$ and let $J$ be the essential ideal of $A$
defined as in \eqref{eq J is ideal} (for our present choice of $U$).
Let $y\in B(\Omegawk)$ be such that $\rho(y)=(y_i)_i$.

For each $\omega\in\Omega$, we have that $y\omega\in\Omegawk$.

\bigskip

{\sc Claim 1. }If $\omega\in \Omega$ is such that $\omega(s)=0$ for all $s\in \Delta\setminus U$,
then $y\omega\in \Omega$ and $y\omega(s)=0$ for $s\in \Delta\setminus U$.

\bigskip

Assuming Claim 1,
consider the set $F_+=\mbox{span}\,\{\Theta_{\omega,\omega}:\
\omega\in \Omega,\ \omega(s)=0 \ \mbox{ for } \ s\in
\Delta\setminus U\}\,$, which by Lemma \ref{lemma:pos dense} is
dense in $K_+$, where $K$ is the essential ideal of $K(\Omega)$
defined by $K=K(\Omega)\cap J$. By the Claim,
$y\Theta_{\omega,\omega}= \Theta_{y\omega,\omega}\in K$ for all
$\omega\in\Omega$. Therefore, $y$ is a left multiplier of $K$.
Similarly, $y$ is a right multiplier of $K$ , which yields $y\in
M(K)$. Hence, $(x_i)_{i\in I}$ is within $\varepsilon$ of a
multiplier---namely, $\rho(y)$---of an essential ideal of
$\rho\left(K(\Omega)\right)$. Thus, by the Frank--Paulsen
description of local multiplier algebras \cite{frank--paulsen2003},
$(x_i)_{i\in I}\in\rho\left(\mloc
(K(\Omega))\right)$. By Theorem \ref{theorem:extension of
Somerset}, $\mloc(A)=\mloc(K(\Omega))$, so $(x_i)_{i\in I}\in\rho(\mloc(A))$.

We are now left with proving Claim 1.
Assume that $\omega\in \Omega$ with $\omega(s)=0$ for all $s\in \Delta\setminus U$.
Let $i\in I$ and let $\omega_i=\omega|_{\Delta_i}\in \Omega_i$.
Note that
for every $\eta_i\in \Omega_i$, $\Theta_{\omega_i,\,\eta_i}\in J_i$, and hence
$\Theta_{y_i\omega_i,\eta_i}=y_i\Theta_{\omega_i,\eta_i}\in J_i\,$.
Also, $y_i\omega_i\in \Omega_i$. Indeed,
suppose that $s_0\in \Delta_i$ and let $\eta_i\in \Omega_i$ such that $\|\eta_i(s_0)\|=1$. Choose a clopen subset $V_i\subset \Delta_i$ of $s_0$ for which
$\|\eta_i(s)\|\geq 1/2$ for all $s\in V_i$
and define
$f(s)\,=\,\chi_{V_i}(s)\|\eta_i(s)\|^{-2}\,$.
Thus, $f\in C(\Delta_i)$ and so $f\cdot \eta_i\in \Omega_i\,$. Then,
since $\Theta_{y_i\omega_i,\eta_i}\in J_i\subset A_i\,$,
we have $\Theta_{y_i\omega_i,\eta_i}(f\cdot\eta_i)\in\Omega_i\,$. So
$\chi_{V_i}\cdot y_i \omega_i=\Theta_{y_i\omega_i,\eta_i}(f\cdot\eta_i)
 \in \Omega_i\,$.
Thus, $y_i\,\omega_i$ is a local uniform limit of vectors fields in $\Omega_i$
and hence, $y_i\,\omega_i\in \Omega_i$.
Moreover, since $\Theta_{y_i\omega_i, \, \eta_i}\in J_i$ for any $\eta_i\in \Omega_i$,
we have $y_i\omega_i(s)=0$ for $s\in \Delta_i\setminus U_i$.

Since $(y\omega)(s)=(y_i\,\omega_i)(s)$ for $s\in\Delta_i\,$, the lower semicontinuous function
$s\mapsto \|(y\omega)(s)\|$ is continuous on $\bigcup_i \Delta_i$ and vanishes
on $(\bigcup_i \Delta_i)\setminus U$.

{\sc Claim 2. }There exists $C>0$ such that $\|y\omega(s)\|\leq C\,\|\omega(s)\|\,$,
$s\in \Delta_i$, $i\in I$.

We will use Claim 2 to show that the function $s\mapsto \|(y\omega)(s)\|$ is continuous on $\Delta$.
Let $s\in \Delta\setminus (\bigcup_i \Delta_i)$ and let
$(s_\alpha)_{\alpha}\subset \bigcup_i \Delta_i$  be a net such that
$s_\alpha\rightarrow s$ in $\Delta$. This implies that
$\lim_\alpha \|\omega(s_\alpha)\|=0$. By lower semicontinuity of the function $s\mapsto \|(y\omega)(s)\|$,
\[
0\leq \|y\omega(s)\|\leq \lim_\alpha\|y\omega(s_\alpha)\|\leq C\ \lim_\alpha\|\omega(s_\alpha) \|=0\,,
\]
and it follows that
$s\mapsto \|(y\omega)(s)\|$ is continuous on $\Delta$ and vanishes in $\Delta\setminus U$.  This establishes Claim 1.

We finish the proof by proving Claim 2. Fix $s\in\Delta_i$, and let
$C=\sup_i\|y_i\|$. We already know that $y_i\omega_i\in\Omega_i$, and so
\begin{align*}
\|y\omega(s)\|&=\|y_i\omega_i(s)\|=\|y_i\omega_i\|(s)\leq\|y_i\|\,\|\omega_i\|(s)\\
&\leq C\,\|\omega_i\|(s)
=C\,\|\omega_i(s)\|=C\,\|\omega(s)\|. \qedhere
\end{align*}
\end{proof}

Local multiplier algebras behave well under direct sums: $\mloc(\oplus_i A_i)\cong\oplus_i\,\mloc(A_i)$ \cite[Proposition 2.3.6]{Ara--Mathieu-book}. However,
the isomorphism of local multiplier algebras in Theorem \ref{theorem:reduction theorem-2}
cannot be established via that generic result:
\begin{proposition}\label{remark:examples 2}  Assume the notation, hypotheses, and conclusions of Theorem \ref{theorem:reduction theorem-2}.
Although $\rho$ sends $A$ into $\bigoplus_iA_i$, it need not be true that
$A\cong\bigoplus_iA_i$.
\end{proposition}
\begin{proof}
If $\Delta$ and $\Omega$ are as in Proposition \ref{proposition:examples 1}, then
$\rho(\Theta_{\omega,\omega})=(\Theta_{\omega_i,\omega_i})_{i\in I}\in \oplus _{i\in I} A_i$, but
$\rho(\Theta_{\omega,\omega})\not\in\rho(A)$.
\end{proof}

%%%%%%%%%%%%%%%%%%%%%%%%%%% bibliography %%%%%%%%%%%%%%%%%%%%%%%%%


\begin{thebibliography}{10}

\bibitem{ara-mathieu1999}
P.~Ara and M.~Mathieu.
\newblock A simple local multiplier algebra.
\newblock {\em Math. Proc. Cambridge Philos. Soc.}, 126(3):555--564, 1999.

\bibitem{Ara--Mathieu-book}
P.~Ara and M.~Mathieu.
\newblock {\em Local multipliers of {$C\sp *$}-algebras}.
\newblock Springer Monographs in Mathematics. Springer-Verlag London Ltd.,
  London, 2003.

\bibitem{ara-mathieu2008}
P.~Ara and M.~Mathieu.
\newblock Maximal {$C^*$}-algebras of quotients and injective envelopes of
  {$C^*$}-algebras.
\newblock {\em Houston J. Math.}, 34(3):827--872, 2008.

\bibitem{argerami--farenick--massey2009}
M.~Argerami, D.~Farenick, and P.~Massey.
\newblock The gap between local multiplier algebras of {$C^*$}-algebras.
\newblock {\em Q. J. Math.}, 60(3):273--281, 2009.

\bibitem{Berberian-book}
S.~K. Berberian.
\newblock {\em Baer *-rings}.
\newblock Springer-Verlag, New York, 1972.
\newblock Die Grundlehren der mathematischen Wissenschaften, Band 195.

\bibitem{birkenmeier-park-rizvi2009}
G.~F. Birkenmeier, J.~K. Park, and S.~T. Rizvi.
\newblock Hulls of semiprime rings with applications to {$C\sp *$}-algebras.
\newblock {\em J. Algebra}, 322(2):327--352, 2009.

\bibitem{blecher--paulsen2001}
D.~P. Blecher and V.~I. Paulsen.
\newblock Multipliers of operator spaces, and the injective envelope.
\newblock {\em Pacific J. Math.}, 200(1):1--17, 2001.

\bibitem{dixmier--douady1963}
J.~Dixmier and A.~Douady.
\newblock Champs continus d'espaces hilbertiens et de {$C\sp{\ast}
  $}-alg\`ebres.
\newblock {\em Bull. Soc. Math. France}, 91:227--284, 1963.

\bibitem{dumitru--peligrad--visinescu2006}
R.~Dumitru, C.~Peligrad, and B.~Vi{\c{s}}inescu.
\newblock Automorphisms inner in the local multiplier algebra and {C}onnes
  spectrum.
\newblock In {\em Operator theory 20}, volume~6 of {\em Theta Ser. Adv. Math.},
  pages 75--80. Theta, Bucharest, 2006.

\bibitem{fell1961}
J.~M.~G. Fell.
\newblock The structure of algebras of operator fields.
\newblock {\em Acta Math.}, 106:233--280, 1961.

\bibitem{frank--paulsen2003}
M.~Frank and V.~I. Paulsen.
\newblock Injective envelopes of {$C\sp *$}-algebras as operator modules.
\newblock {\em Pacific J. Math.}, 212(1):57--69, 2003.

\bibitem{halpern1966}
H.~Halpern.
\newblock The maximal {${\rm GCR}$} ideal in an {${\rm AW}\sp{\ast} $}-algebra.
\newblock {\em Proc. Amer. Math. Soc.}, 17:906--914, 1966.

\bibitem{hamana1979a}
M.~Hamana.
\newblock Injective envelopes of {$C\sp{\ast} $}-algebras.
\newblock {\em J. Math. Soc. Japan}, 31(1):181--197, 1979.

\bibitem{hamana1981}
M.~Hamana.
\newblock Regular embeddings of {$C\sp{\ast} $}-algebras in monotone complete
  {$C\sp{\ast} $}-algebras.
\newblock {\em J. Math. Soc. Japan}, 33(1):159--183, 1981.

\bibitem{hamana-tensor}
M.~Hamana.
\newblock Tensor products for monotone complete {$C\sp{\ast} $}-algebras. {I},
  {II}.
\newblock {\em Japan. J. Math. (N.S.)}, 8(2):259--283, 285--295, 1982.

\bibitem{kaplansky1953}
I.~Kaplansky.
\newblock Modules over operator algebras.
\newblock {\em Amer. J. Math.}, 75:839--858, 1953.

\bibitem{Lance-book}
E.~C. Lance.
\newblock {\em Hilbert {$C\sp *$}-modules}, volume 210 of {\em London
  Mathematical Society Lecture Note Series}.
\newblock Cambridge University Press, Cambridge, 1995.
\newblock A toolkit for operator algebraists.

\bibitem{Paulsen-book}
V.~Paulsen.
\newblock {\em Completely bounded maps and operator algebras}, volume~78 of
  {\em Cambridge Studies in Advanced Mathematics}.
\newblock Cambridge University Press, Cambridge, 2002.

\bibitem{pedersen1978}
G.~K. Pedersen.
\newblock Approximating derivations on ideals of {$C\sp*$}-algebras.
\newblock {\em Invent. Math.}, 45(3):299--305, 1978.

\bibitem{pedersen1984}
G.~K. Pedersen.
\newblock Multipliers of {${\rm AW}\sp{\ast} $}-algebras.
\newblock {\em Math. Z.}, 187(1):23--24, 1984.

\bibitem{somerset1996}
D.~W.~B. Somerset.
\newblock The local multiplier algebra of a {$C\sp *$}-algebra.
\newblock {\em Quart. J. Math. Oxford Ser. (2)}, 47(185):123--132, 1996.

\bibitem{somerset2000}
D.~W.~B. Somerset.
\newblock The local multiplier algebra of a {$C\sp *$}-algebra. {II}.
\newblock {\em J. Funct. Anal.}, 171(2):308--330, 2000.

\bibitem{takemoto1976}
H.~Takemoto.
\newblock On the weakly continuous constant field of {H}ilbert space and its
  application to the reduction theory of von {N}eumann algebra.
\newblock {\em T\^ohoku Math. J. (2)}, 28(3):479--496, 1976.

\bibitem{Takesaki-bookI}
M.~Takesaki.
\newblock {\em Theory of operator algebras. {I}}, volume 124 of {\em
  Encyclopaedia of Mathematical Sciences}.
\newblock Springer-Verlag, Berlin, 2002.
\newblock Reprint of the first (1979) edition, Operator Algebras and
  Non-commutative Geometry, 5.

\end{thebibliography}
\end{document}